\documentclass[sn-mathphys,Numbered]{sn-jnl}

\usepackage{graphicx}%
\usepackage{multirow}%
\usepackage{amsmath,amssymb,amsfonts}%
\usepackage{amsthm}%
\usepackage{mathrsfs}%
\usepackage[title]{appendix}%
\usepackage{xcolor}%
\usepackage{textcomp}%
\usepackage{manyfoot}%
\usepackage{booktabs}%
\usepackage{algorithm}%
\usepackage{algorithmicx}%
\usepackage{algpseudocode}%
\usepackage{listings}%

\newcommand{\E}{\mathbb{E}}  
\newcommand{\R}{\mathbb{R}}  
\newcommand{\set}[1]{\left\{#1\right\}}

\newcommand{\norms}[1]{\Vert#1\Vert}
\newcommand{\dom}[1]{\mathrm{dom}\left(#1\right)}
\DeclareMathOperator{\Ocal}{\mathcal{O}} 
\newtheorem{thm}{Theorem}
\newtheorem{lem}{Lemma}

\newtheorem{cor}{Corollary}

\newtheorem{rem}{Remark}
\newtheorem{ass}{Assumption}
\begin{document}

\title[Shuffling Momentum Gradient Algorithm for Convex Optimization]{Shuffling Momentum Gradient Algorithm for Convex Optimization}








\author[1]{\fnm{Trang H.} \sur{Tran}}\email{htt27@cornell.edu}

\equalcont{Corresponding author}

\author[2]{\fnm{Quoc} \sur{Tran-Dinh}}\email{quoctd@email.unc.edu}

\author[3]{\fnm{Lam M.} \sur{Nguyen}}\email{LamNguyen.MLTD@ibm.com}

\affil[1]{\orgdiv{School of Operations Research and Information Engineering}, \orgname{Cornell University}, \orgaddress{Ithaca, NY, USA}}

\affil[2]{\orgdiv{Department of Statistics and Operations Research}, \orgname{The University of North Carolina at Chapel Hill}, \orgaddress{Chapel Hill, NC, USA}}

\affil[3]{\orgdiv{IBM Research}, \orgname{Thomas J. Watson Research Center}, \orgaddress{Yorktown Heights, NY, USA}}


\abstract{ 
The Stochastic Gradient Descent method (SGD) and its stochastic variants have become methods of choice for solving finite-sum optimization problems arising from machine learning and data science thanks to their ability to handle large-scale applications and big datasets.
In the last decades, researchers have made substantial effort to study the theoretical performance of SGD and its shuffling variants.
However, only limited work has investigated its shuffling momentum variants, including shuffling heavy-ball momentum schemes for non-convex problems and Nesterov's momentum for convex settings. 
In this work, we extend the analysis of the shuffling momentum gradient method developed in \cite[Tran et al (2021)]{smg_tran21b} to both finite-sum convex and strongly convex optimization problems. 
We provide the first analysis of shuffling momentum-based methods for the strongly convex setting, attaining a convergence rate of $\Ocal(1/nT^2)$, where $n$ is the number of samples and $T$ is the number of training epochs. 
Our analysis is a state-of-the-art, matching the best rates of existing shuffling stochastic gradient algorithms in the literature. 
}




\maketitle


\section{Introduction}\label{sec:intro}
We consider the following common  finite-sum convex optimization problem:
\begin{align}\label{ERM_problem_01}
    \min_{w \in \mathbb{R}^d} \Big\{ F(w) := \frac{1}{n} \sum_{i=1}^n f(w ; i)  \Big\},
\end{align}
where $f(\cdot; i) : \R^d\to\R$ is a given smooth and convex function for $i \in [n] := \set{1,\cdots,n}$. 
Throughout this paper, we assume that the solution set of \eqref{ERM_problem_01} is nonempty.

Problem \eqref{ERM_problem_01} can be viewed as a sample average approximation of a stochastic optimization problem and frequently arises in machine learning tasks such as linear least-squares, logistic regression, and multi-kernel learning. 
The primary challenge in addressing \eqref{ERM_problem_01} stems from the high-dimensional setting and/or involving a large training dataset, often corresponding to the number of component functions $n$. 
Consequently, deterministic methods that require full gradient evaluations are typically inefficient for solving this type of problems \citep{sra2012optimization, Bottou2018}.

\textbf{Stochastic gradient method and its shuffling variants.}
The stochastic gradient descent method (SGD) \cite{RM1951} and its stochastic first-order variants have become a preferred optimization technique for solving \eqref{ERM_problem_01}, due to their scalability and efficiency in addressing large-scale models  \citep{AdaGrad,Kingma2014,Bottou2018,Nguyen2018_sgdhogwild}.
At each iteration, SGD samples an index $i$ (or a subset of indices) uniformly from $\{1, \dots, n\}$ and utilizes the stochastic gradient $\nabla f(\cdot; i)$ to update the weights  or model parameters. 
While the uniformly independent sampling of the index $i$ plays a pivotal role in our theoretical understanding of SGD, practical heuristics often employ without-replacement sampling schemes, also known as shuffling sampling schemes.

Shuffling gradient-type methods rely on random or deterministic permutations of the index set $\{1, 2, \dots, n\}$ and apply incremental gradient updates using these permutation orders. 
A collection of such $n$ individual updates is termed an epoch or a pass over all the training data set. 
One of the most popular methods in this category is the Random Reshuffling scheme, which shuffles a new random permutation at the beginning of each epoch. 
Other notable methods include the Single Shuffling (using one random permutation for each epoch) and the Incremental Gradient (employing a deterministic order of the indices) algorithms. 

Empirical studies indicate that shuffling sampling schemes often lead to faster convergence than SGD \citep{bottou2009curiously}. However, due to the absence of statistical independence, analyzing these shuffling variants is often more challenging than the identically distributed version. Recent works have shown theoretical improvements for shuffling schemes over SGD in terms of the number of epochs $T$ \citep{Gurbuzbalaban2019,haochen2019random,safran2020good,nagaraj2019sgd,rajput2020closing,nguyen2020unified,mishchenko2020random, ahn2020sgd}.
For example, in a strongly convex setting, shuffling sampling schemes enhance the convergence rate of SGD from $\Ocal(1/{T})$ to $\Ocal(1/nT^2)$, where $T$ is the number of effective data passes \citep{nguyen2020unified,mishchenko2020random, reddi2019convergence}.

\textbf{Stochastic momentum methods.} 
Since the success of SGD for large-scale stochastic optimization observed in the machine learning community, researchers have actively studied stochastic variants of momentum methods as these methods also rely on first-order information akin to SGD \citep{NIPS_hu2009,NIPS_cotter2011,AdaGrad,Kingma2014}. 
Extensive research has been conducted to studying stochastic heavy-ball method \citep{Polyak1964,pmlr-v28-sutskever13}, Nesterov's accelerated gradient \cite{Nesterov1983, Nesterov2004,NIPS_hu2009, NIPS_cotter2011,JMLR_yuan2016} and other adaptive step-size variants \citep{AdaGrad,Kingma2014, jin2022convergence, liu2019variance}. 
However, the majority of these works often consider uniformly sample data, with the assumption of unbiased gradients. There have been very few work in the shuffling literature focusing on the heavy-ball momentum version of SGD for the non-convex settings \cite{smg_tran21b} and Nesterov's momentum in the convex settings \cite{Tran2022_ShufflingNesterov}. 

\textbf{Contributions.} 
In this paper, we revisit the Shuffling Momentum Gradient (SMG) algorithm developed in our previous work \cite{smg_tran21b}.
This scheme can be viewed as a variant of the Shuffling SGD algorithm with the application of an anchor momentum. 
While our prior work provides the analyses of SMG for non-convex settings, we investigate this algorithm here for both convex and strongly convex settings, filling the gap between the two worlds. 
These analyses help broaden and enhance our understanding of shuffling gradient methods and their momentum variants across diverse settings. 
It is worth noting that our analysis is the first work that analyzes shuffling momentum schemes for a strongly convex setting, attaining the convergence rate of $\Ocal(1/nT^2)$, where $n$ is the number of samples and $T$ is the number of training epochs. This analysis, along with our convex assumption, matches the state-of-the-art convergence rate of Shuffling SGD algorithm in both strongly convex and merely convex settings. 

\textbf{Related work.} 
Let us review various related works to our method and our analysis in this paper, especially works on shuffling gradient and momentum algorithms.

\textbf{\textit{Shuffling Gradient Methods.}}
In the big data era, Random Reshuffling have been more favorable than plain SGD due to their superior practical performance and straightforward implementation \citep{bottou2009curiously,bottou2012stochastic, Recht2011}. The key challenge in the theoretical analysis of randomized shuffling schemes is the absence of conditionally unbiased gradients: $\E \left [\nabla f(y_i^{(t)}; \pi_i^{(t)}) \right] \neq \nabla F(y_i^{(t)})$, where $t$ is the current epoch. 
Recent advancements in the analysis of shuffling techniques include \citep{Gurbuzbalaban2019,haochen2019random,safran2020good,nagaraj2019sgd, ahn2020sgd}. 
Most of these works focus on the strongly convex case, assuming either a bounded gradient or a bounded domain. The best-known convergence rate in this scenario is $\Ocal(1/(nT)^2 + 1/(nT^3))$, where $T$ is the number of epochs. 
This result matches the lower bound rate in \citep{safran2020good} up to a constant factor.

In the convex regime, there have been early attempts to analyze the deterministic Incremental Gradient scheme \citep{nedic2001convergence,nedic2001incremental}. 
More recent works explore convergence theory for various shuffling schemes \citep{shamir2016without,mishchenko2020random,nguyen2020unified}, with \citet{nguyen2020unified} offering a unified approach to different shuffling schemes and proving a convergence rate of $\Ocal(1/ T^{2/3})$. 
When a randomized scheme is applied (Random Reshuffling or Single Shuffling), the bound in expectation improves to $\Ocal(1/T + 1/(n^{1/3} T^{2/3}))$. 
In this paper, we conduct a convex analysis of Algorithm \ref{sgd_momentum_shuffling_01} (see Section~\ref{sec:algo}) with the convergence rate of $\Ocal(1/(n^{1/3} T^{2/3}))$ under standard assumptions for the same setting.

\textbf{\textit{Shufling Momentum Methods.}}
While significant progress has been made in analyzing the shuffling variant of SGD, there has been limited work on the shuffled adaptation of well-known momentum methods, such as the heavy ball method or adaptive step size of Adam-type algorithms. 
The first work attempting to analyze shuffling momentum methods is our previous work \cite{smg_tran21b}, which studies the heavy ball version in a simple nonconvex setting under a bounded gradient assumption. 
\cite{smg_tran21b} introduces a variant of the heavy ball momentum which aligns well with shuffling methods, named by the SMG algorithm. 
SMG has convergence guarantees in various nonconvex settings, including a bounded variance case and different data shuffling strategies. 
In the convex regime, the work \cite{Tran2022_ShufflingNesterov} suggests the application of Nesterov's accelerated momentum technique for shuffling gradient methods and then demonstrates an improved deterministic convergence rate of $\Ocal(1/T)$ for unified shuffling schemes. 

Alternatively, a popular line of research focuses on variance reduction techniques, which have demonstrated promising performance for convex optimization (e.g., SAG \citep{SAG}, SAGA \citep{SAGA}, SVRG \citep{SVRG}, and SARAH \citep{Nguyen2017sarah}). Typically, these methods involve computing or storing a complete gradient or a substantial batch of gradients which is a crucial step in minimizing variance.  
Since the updates of SGD, Shuffling SGD, and the SMG algorithm do not compute the full gradient at any stage, consequently, the SMG algorithm aligns with the category of shuffling SGD-based methods, which is different from variance reduction methods.

\section{Shuffling Momentum Gradient Algorithm}\label{sec:algo}
In this section, we briefly describe our Shuffling Momentum Gradient (SMG Algorithm) from solving \eqref{ERM_problem_01}.
This method was first developed in our work \cite{smg_tran21b} for a nonconvex setting. 
In this paper, we apply it to solve both the convex and strongly convex settings.
Though the algorithm is the same as in \cite{smg_tran21b}, its convergence analysis is completely different and new.
Let us first present our SMG as in Algorithm~\ref{sgd_momentum_shuffling_01} below.

\begin{algorithm}[hpt!]
\caption{SMG Algorithm - Shuffling Momentum Gradient from \cite{smg_tran21b}}
\label{sgd_momentum_shuffling_01}
\begin{algorithmic}[1]
\State {Initialization:} 
Choose $\tilde{w}_0 \in\R^d$ and set $\tilde{m}_0 := \textbf{0}$.{\!\!}
\For{$t := 1,2,\cdots,T $}
\State Set $w_0^{(t)} := \tilde{w}_{t-1}$; $m_0^{(t)} := \tilde{m}_{t-1}$; and $v_0^{(t)} := \textbf{0}$. \label{alg:A1_step1}
\State Generate an arbitrarily deterministic or random permutation  $\pi^{(t)}$ of $[n]$. \label{alg:A1_step2}
\For{$i := 0,\cdots,n-1$}
\State Query $ g_{i}^{(t)} := \nabla f ( w_{i}^{(t)} ; \pi^{(t)} ( i + 1 ) )$. \label{alg:A1_step3}
\State Choose $\eta^{(t)}_i := \frac{\eta_t}{n}$ and update \label{alg:A1_step4}
\begin{equation}\label{eq:main_update}
\left\{\begin{array}{lcl}
m_{i+1}^{(t)} & := & \beta m_{0}^{(t)} + (1-\beta) g_{i}^{(t)} \vspace{1ex}\nonumber\\
v_{i+1}^{(t)} & := &  v_{i}^{(t)} + \frac{1}{n} g_{i}^{(t)} \vspace{1ex}\nonumber\\
w_{i+1}^{(t)} & := & w_{i}^{(t)} - \eta_i^{(t)} m_{i+1}^{(t)}.
\end{array}\right.
\end{equation}
\EndFor
\State Set $\tilde{w}_t := w_{n}^{(t)}$ and $\tilde{m}_t := v_{n}^{(t)}$. \label{alg:A1_step5}
\EndFor
\State\textbf{Output:} 
Choose $\hat{w}_T \in \{ \tilde{w}_{0},\cdots,\tilde{w}_{T-1}\}$ at random with probability $\mathbb{P}[\hat{w}_T = \tilde{w}_{t-1}] = \frac{\eta_t}{\sum_{t=1}^{T} \eta_t}$.
\label{alg:A1_step6}
\end{algorithmic}
\end{algorithm}

Let us highlight the following points.
In contrast to existing momentum methods in the literature, where the momentum term $m_i^{(t)}$ is recursively updated as $m_{i+1}^{(t)} := \beta m_i^{(t)} + (1-\beta) g_i^{(t)}$ for $\beta \in (0, 1)$, our SMG  adopts a different approach from ``anchored methods". It keeps the first constant term $m_0^{(i)}$ in the primary update at each epoch. 
This momentum term (or the anchor point) is updated solely at the end of each epoch by averaging all the gradient components $\{ g_{i}^{(t)}\}_{i=0}^{n-1}$ evaluated during that epoch. 
To avoid storing $n$ terms $g_i^{(t)}$ at each epoch, SMG employs an auxiliary variable $v_i^{(t)}$ to store the gradient average.
The application of this new momentum update is based on two observations.
Firstly, updating the momentum after each epoch aids with the shuffling data schemes where summing over an epoch replicate the full gradient. 
Secondly, $m_0^{(t)}$ represents an equal-weighted average of all past gradients within an epoch, which is different from the traditional momentum with exponential decay weights. 

It is worth noting that the SMG algorithm reduces to the standard shuffling gradient method \citep{nguyen2020unified,mishchenko2020random} when $\beta = 0$. 
In our analysis, we use $\eta_i^{(t)} = \frac{\eta_t}{n}$ in Algorithm~\ref{sgd_momentum_shuffling_01}, which is consistent with prior analysis for shuffling gradient methods, e.g., in  \citep{nguyen2020unified,smg_tran21b,Nguyen2022_SGDPL,Tran2022_ShufflingNesterov}. 
Our step size matches the step size in \cite{mishchenko2020random} with the same order of training samples $n$ in the corresponding settings and state-of-the-art results. 
We discuss more details of our learning rate $\eta_t$ in our theoretical analysis below. 

Note that Algorithm~\ref{sgd_momentum_shuffling_01} works with any permutation $\pi^{(t)}$ of $\{1,2,\cdots, n\}$, including deterministic and randomized ones. 
Hence, it covers a wide range of shuffling methods.
Our early work \cite{smg_tran21b} analyzes most of results for unified shuffling schemes for nonconvex problems which includes incremental, single shuffling, and randomized reshuffling variants as special cases. 
Note that the convergence results for the randomized reshuffling variant is usually better than the general ones due to the fact that randomized reshuffling allows for independence between the epochs of the algorithm, thus leads to better bounds in expectation of the gradient variance at the solution \cite{mishchenko2020random, nguyen2020unified, smg_tran21b}. Their analyses involving the gradient variance are often similar. As a consequence, in this paper, we only present the results for randomized reshuffling, the unified results for other variants follows in a similar manner. 


\section{{Technical Assumptions and Key Bounds}}\label{sec:assumptions}
In this section, we present the technical assumptions and prove necessary key bounds used in our subsequent analyses. 

\subsection{Technical Assumptions}\label{subsec:assumptions}
\begin{ass}[\textbf{Smoothness}]\label{as:A1}
For each $i \in [n]$, the component function $f(\cdot; i)$ is $L$-smooth, i.e. there exists a universal smoothness constant $L > 0$ such that, for all $w, w' \in \dom{F}$, it holds that
\begin{equation}\label{eq:Lsmooth_basic}
\norms{ \nabla f(w;i) - \nabla f(w';i)} \leq L \norms{ w - w'}. 
\end{equation}
\end{ass}  
The $L$-smoothness \eqref{eq:Lsmooth_basic} assumption is standard in the analyses of gradient-type methods for both stochastic and deterministic algorithms. From this assumption, we have the following bound for any $w, w' \in \dom{F}$ \citep{Nesterov2004}:
\begin{align*}
F(w) \leq F(w') + \langle\nabla F(w'), w-w'\rangle + \frac{L}{2}\norms{w-w'}^2. 
\end{align*}

Since we study a convex setting of \eqref{ERM_problem_01}, we impose the following assumption.
\begin{ass}[Convexity]
\label{ass_convex}
{For each $i \in [n]$, the component function $f(\cdot; i)$ of \eqref{ERM_problem_01} is proper and convex, i.e.:
\begin{align*}
f(w;i) \geq f(\hat{w}; i) + \langle\nabla f(\hat{w};i), w - \hat{w}\rangle,\quad \forall w, \hat{w} \in \dom{f(\cdot; i)}.
\end{align*}
}
\end{ass}
In addition, we also investigate a strongly convex case of \eqref{ERM_problem_01} as stated in the following assumption. While individual convexity is needed for our analysis, we do not requires the strongly convexity of individual functions.
\begin{ass}[$\mu$-strong convexity]
\label{ass_stronglyconvex}
The objective function $F$ of \eqref{ERM_problem_01}  is $\mu$-strongly convex on $\dom{F}$, i.e. there exists a constant $\mu \in (0, +\infty)$ such that
\begin{align*}
F(w) \geq F(\hat{w}) + \langle\nabla F(\hat{w}), w - \hat{w}\rangle + \frac{\mu}{2}\norms{w - \hat{w}}^2,\quad \forall w, \hat{w} \in \dom{F}.
\end{align*}
\end{ass}
\subsection{Basic Notations}\label{subsec:notations}
The key derivations of our analysis consider the setting where $f(\cdot;i)$ is $L$-smooth and convex for all $i \in [n]$. 
Before stating our results, we need some basic notations. 
Together with the conventional term $g_{i}^{(t)} := \nabla f ( w_{i}^{(t)} ; \pi^{(t)} ( i + 1 ) )$  for gradient, we denote the following terms:
\begin{align*}
    & f_i^{(t)} (\cdot) := f ( \cdot  ; \pi^{(t)} ( i + 1 ) ),  \\
    & D_i^{(t)} (w_1, w_2)  := D_{f_i^{(t)}} (w_1, w_2) = f_i^{(t)} (w_1)- f_i^{(t)} (w_2) -  \langle \nabla f_i^{(t)} (w_2),  w_1 - w_2 \rangle ,
\end{align*}
where $f_i^{(t)} (\cdot)$ is the component function that has index $\pi^{(t)} ( i + 1 ) $ (chosen at the outer iteration $t$ and the inner iteration $i$), and $ D_i^{(t)} (w_1, w_2)$ is the Bregman
divergence between $w_1$ and $w_2$ associated with $f_i^{(t)}$.
If $f(\cdot;i)$ is $L$-smooth and convex, then for all $w_1, w_2 \in \mathbb{R}^d$, we have the following inequality \citep{Nesterov2004}:
\begin{align*}
    f(w_1; i) \geq f(w_2; i) + \langle\nabla f(w_2; i), w_1 - w_2\rangle + \frac{1}{2 L}\norms{\nabla f(w_2; i) - \nabla f(w_1; i)}^2,
\end{align*}
which leads to 
\begin{align}
    D_i^{(t)} (w_1, w_2)  &\geq \frac{1}{2 L}\norms{\nabla f_i^{(t)}(w_1) - \nabla f_i^{(t)}(w_2)}^2, \quad \forall w_1, w_2 \in \mathbb{R}^d, \label{eq_div_1} \\
    D_{i}^{(t)}(w_1, w_2) &\leq \frac{L}{2} \left\|w_1- w_2\right\|^2, \quad  \forall w_1, w_2 \in \mathbb{R}^d.\label{eq_div_2}
\end{align}
In our analysis below, we repeatedly use the following quantity, which defines the total Bregman divergence between $w_*$ and $w_{j}^{(t)}$ in each epoch $t$, i.e.
\begin{align}
    C_0 := 0, \quad C_t := \sum_{j=0}^{n-1} D_{j}^{(t)}(w_*; w_{j}^{(t)}), \quad  t \geq 1. \label{define_C} 
\end{align}
Another quantity is the following sequence of expected values:
\begin{align}
    F_t :=\beta \cdot \E[C_{t-1}] + (1-\beta) \cdot \E[C_t] \quad  t \geq 1. \label{define_F} 
\end{align}
Since the Algorithm \ref{sgd_momentum_shuffling_01} sets $\tilde{m}_0 := \textbf{0}$ at the first iteration, the analysis of the first iteration is slightly different from other iterations, e.g., $F_t = (1-\beta) \cdot \E[C_t]$ for $t=1$ and we adopt the convention that $C_0 = 0$. Finally, let $w_{*}$ be an optimal solution of \eqref{ERM_problem_01}.
We define the variance $\sigma^2$ of $F$ at $w_{*}$ as follows:
\begin{align}\label{defn_finite}
\sigma^2 := \frac{1}{n}\sum_{i=1}^{n}\Vert \nabla{f}(w_{*}; i) \Vert^2  \in [0, +\infty).
\end{align}
\subsection{Key Derivations}\label{subsec:derivations}

Now, we are ready to state our key bound used in our subsequence analysis.
In fact, the following lemma bounds the quantity of interest, which is the expected objective residual $\E \left[ F( \tilde{w}_{t}) - F(w_*)\right]$ based on the expected distance $\E \left[\|\tilde{w}_{t-1} - w_*\|^2 \right]$ from the current iterate to an optimal solution and the quantity $F_t$ defined by \eqref{define_F}.

\begin{lem}
\label{lem_smg_convex}
Suppose that Assumption \ref{as:A1} and Assumption~\ref{ass_convex} holds for \eqref{ERM_problem_01}.
Let $\{w_i^{(t)}\}_{t=1}^{T}$ be generated by Algorithm \ref{sgd_momentum_shuffling_01} with a fixed momentum parameter $0\leq \beta < 1$ and an epoch learning rate $\eta_i^{(t)} := \frac{\eta_t}{n}$ for $t \geq 1$. 
Assume that $0 < \eta_t \leq \frac{1}{2L\sqrt{K}}$ for $t \geq 1$, where $K \geq 1$. {Suppose further that a randomized reshuffling strategy is used in Algorithm \ref{sgd_momentum_shuffling_01}. }
Then, for $t \geq 2$, we have the following bound:
\begin{align}\label{eq:key_bound1}
     &2 \eta_t \E \left[ F( \tilde{w}_{t}) - F(w_*)\right] \leq  \E \left[\|\tilde{w}_{t-1} - w_*\|^2 \right] -  \E \left[\left\|\tilde{w}_{t}- w_*\right\|^2  \right] - \frac{2\eta_t}{n} F_t \nonumber\\
     &\quad + \frac{2\eta_t}{n} \frac{F_t}{K}  + \frac{2 \beta \eta_t}{n} \frac{ F_{t-1} }{K}  + \frac{4L\sigma^2 }{3n}  \beta(1-\beta)  \eta_t\eta_{t-1}^2  +  \frac{4L\sigma^2}{3n} (1-\beta)^2\eta_t^3.
\end{align}
If $t=1$, then we have: 
\begin{align*}
    2 \eta_t(1-\beta)\E \left[ F( \tilde{w}_{t} ) - F(w_*) \right]  &\leq  \E \left[\|\tilde{w}_{t-1} - w_*\|^2 \right] -  \E \left[\left\| \tilde{w}_{t} - w_*\right\|^2  \right] - \frac{2\eta_t}{n} F_t  \nonumber\\
    & \quad + (1-\beta) \frac{2\eta_t}{n} \frac{1}{K} F_t +  \frac{4L\sigma^2}{3n} (1-\beta)^2\eta_t^3.
\end{align*}
\end{lem}
Since the proof of Lemma~\ref{lem_smg_convex} is relatively technical, we defer it to the appendix. 
In the next section, we will present our main theoretical results for two settings: convex and strongly convex cases. 

\section{Theoretical Convergence Results}\label{sec:analysis}
We analyze the convergence of Algorithm~\ref{sgd_momentum_shuffling_01} in the two settings: merely convex and strongly convex cases as follows.

\subsection{Convex Case: Main Result and Consequences}
In this subsection, we consider the setting where $f(\cdot;i)$ is $L$-smooth and convex for all $i \in [n]$. 
The following theorem  states the convergence of Algorithm~\ref{sgd_momentum_shuffling_01} for this case.

\begin{thm}
\label{thm_smg_convex}
Suppose that Assumption \ref{as:A1} and Assumption~\ref{ass_convex} hold for \eqref{ERM_problem_01}.
Let $\{w_i^{(t)}\}_{t=1}^{T}$ be generated by Algorithm \ref{sgd_momentum_shuffling_01} with a fixed momentum parameter $0\leq \beta < 1$ and an epoch learning rate $\eta_i^{(t)} := \frac{\eta_t}{n}$ for  $t \geq 1$. 
Suppose further that a randomized reshuffling strategy is used in Algorithm \ref{sgd_momentum_shuffling_01}. 
Assume that $\eta_t \leq \alpha \eta_{t-1}$ for some $ \alpha > 0$
and $0 < \eta_t \leq \frac{1}{2L\sqrt{K}}$ for $t \geq 1$, where $K := {1+ \alpha \beta}$. 
Then, we have the following bound:
\begin{align} \label{eq_thm_smg_convex}
    \E \left[ F( \hat{w}_T ) - F(w_*) \right] 
    \leq \frac{\E \left[\|\tilde{w}_{0} - w_*\|^2 \right]}{2 (1-\beta)\sum_{t=1}^T \eta_t}  +\frac{4L\sigma^2 (1+ \alpha \beta)}{6n  (1-\beta)}   \frac{\sum_{t=1}^T \eta_t^3  }{\sum_{t=1}^T \eta_t}.
\end{align}

\end{thm}
Theorem \ref{thm_smg_convex} only provides an upper bound on $\E \left[ F( \hat{w}_T ) - F(w_*) \right]$, and we have not yet seen a concrete convergence rate of Algorithm~\ref{sgd_momentum_shuffling_01}. 
We discuss the convergence rate in the subsequent corollaries for different choices of the learning rate.

\begin{proof}
First, summing \eqref{eq:key_bound1} from Lemma \ref{lem_smg_convex} for $t=1, 2, \cdots, T$, we have
\begin{align*}
    &2 (1-\beta)\sum_{t=1}^T \eta_t\E \left[ F( \tilde{w}_{t} ) - F(w_*) \right]  
    \leq  \E \left[\|\tilde{w}_{0} - w_*\|^2 \right] -  \E \left[\left\| \tilde{w}_{T} - w_*\right\|^2  \right] - \frac{2}{n} \sum_{t=1}^T \eta_tF_t \nonumber\\
     & + \frac{2}{n} \frac{1}{K} \sum_{t=2}^T\eta_t F_t + \frac{2}{n}\beta \frac{1}{K} \sum_{t=2}^T\eta_t F_{t-1} + \frac{4L\sigma^2 }{3n}  \beta(1-\beta)  \sum_{t=2}^T\eta_t\eta_{t-1}^2    \nonumber\\
    &  +  \frac{4L\sigma^2}{3n} (1-\beta)^2 \sum_{t=2}^T \eta_t^3 + (1-\beta) \frac{2}{n} \frac{1}{K}\eta_1 F_1 +  \frac{4L\sigma^2}{3n} (1-\beta)^2\eta_1^3.  
\end{align*}
Next, using the fact that $\eta_t \leq \alpha \eta_{t-1}$, we get
\begin{align*}
    &2 (1-\beta)\sum_{t=1}^T \eta_t\E \left[ F( \tilde{w}_{t} ) - F(w_*) \right]  \leq  \E \left[\|\tilde{w}_{0} - w_*\|^2 \right] -  \E \left[\left\| \tilde{w}_{T} - w_*\right\|^2  \right] - \frac{2}{n} \sum_{t=1}^T \eta_tF_t \nonumber\\
     & + \frac{2}{n} \frac{1}{K} \sum_{t=2}^T\eta_t F_t + \frac{2}{n} \alpha \beta \frac{1}{K} \sum_{t=2}^T\eta_{t-1} F_{t-1} + \frac{4L\sigma^2 }{3n} \alpha \beta(1-\beta)  \sum_{t=2}^T\eta_{t-1}^3   \nonumber\\
    &  +  \frac{4L\sigma^2}{3n} (1-\beta)^2 \sum_{t=2}^T \eta_t^3+ (1-\beta) \frac{2}{n} \frac{1}{K}\eta_1 F_1 +  \frac{4L\sigma^2}{3n} (1-\beta)^2\eta_1^3 \nonumber\\
    &\leq  \E \left[\|\tilde{w}_{0} - w_*\|^2 \right] -  \E \left[\left\| \tilde{w}_{T} - w_*\right\|^2  \right] - \frac{2}{n} \sum_{t=1}^T \eta_tF_t \nonumber\\
     & + \frac{2}{n} \frac{1}{K} \sum_{t=2}^T\eta_t F_t + \frac{2}{n} \alpha \beta \frac{1}{K} \sum_{t=1}^{T-1}\eta_{t} F_{t} + \frac{4L\sigma^2 }{3n} \alpha \beta(1-\beta)  \sum_{t=1}^{T-1}\eta_{t}^3 +  \frac{4L\sigma^2}{3n} (1-\beta)^2 \sum_{t=2}^T \eta_t^3  \nonumber\\
    &  + (1-\beta) \frac{2}{n} \frac{1}{K}\eta_1 F_1 +  \frac{4L\sigma^2}{3n} (1-\beta)^2\eta_1^3 \nonumber\\
    &\leq  \E \left[\|\tilde{w}_{0} - w_*\|^2 \right] -  \E \left[\left\| \tilde{w}_{T} - w_*\right\|^2  \right] - \frac{2}{n} \sum_{t=1}^T \eta_tF_t \nonumber\\
     & + \frac{2}{n} \frac{1}{K} \sum_{t=1}^T\eta_t F_t + \frac{2}{n} \alpha \beta \frac{1}{K} \sum_{t=1}^{T-1}\eta_{t} F_{t} + \frac{4L\sigma^2 }{3n} \alpha \beta(1-\beta)  \sum_{t=1}^{T-1}\eta_{t}^3 +  \frac{4L\sigma^2}{3n} (1-\beta)^2 \sum_{t=1}^T \eta_t^3  \nonumber\\
    &\leq  \E \left[\|\tilde{w}_{0} - w_*\|^2 \right] -  \E \left[\left\| \tilde{w}_{T} - w_*\right\|^2  \right] - \frac{2}{n} \sum_{t=1}^T \eta_tF_t \nonumber\\
     & + \frac{2}{n} \frac{1}{K} (1+ \alpha \beta) \sum_{t=1}^T\eta_t F_t  + \frac{4L\sigma^2 }{3n}  (1+ \alpha \beta) \sum_{t=1}^T \eta_t^3.  
\end{align*}
Finally, since $K = 1+ \alpha \beta$, divide both sides of the last estimate by $2 (1-\beta)\sum_{t=1}^T \eta_t$, we arrive at
\begin{align*}
   \E \left[ F( \hat{w}_T ) - F(w_*) \right] & \leq \frac{\sum_{t=1}^T \eta_t\E \left[ F( \tilde{w}_{t} ) - F(w_*) \right]}{\sum_{t=1}^T \eta_t} \nonumber\\
    &
    \leq \frac{\E \left[\|\tilde{w}_{0} - w_*\|^2 \right]}{2 (1-\beta)\sum_{t=1}^T \eta_t}  +\frac{4L\sigma^2 (1+ \alpha \beta)}{6n  (1-\beta)}   \frac{\sum_{t=1}^T \eta_t^3  }{\sum_{t=1}^T \eta_t},
\end{align*}
which exactly proves \eqref{eq_thm_smg_convex}.
\end{proof}

Compared to the analysis of shuffling SGD methods for convex problems, our shuffling momentum method only requires an additional assumption $\eta_t \leq \alpha \eta_{t-1}$ for some $ \alpha > 0$, because it needs to deal with past momentum terms. Since we only requires $\alpha$ to be independent of $T$, this condition is mild and can be enforced in most of the traditional learning rate schemes. 
Next, we derive two direct consequences of Theorem~\ref{thm_smg_convex} that  demonstrate the two well-known step size schemes.

\begin{cor}[\textbf{Constant learning rate}]\label{co:constant_LR}
Let us fix the number of epochs $T \geq 1$, and  choose a constant learning rate $\eta_t := \eta = \frac{\gamma n^{1/3}}{T^{1/3}}$ for some $\gamma > 0$ such that $\frac{\gamma n^{1/3}}{T^{1/3}} \leq \frac{1}{2L\sqrt{1+\beta}}$ for $t\geq 1$ in Algorithm~\ref{sgd_momentum_shuffling_01}.
Then, under the conditions of Theorem~\ref{thm_smg_convex}, we have
\begin{align}\label{eq:convergence_rate1}
    & \E \left[ F( \hat{w}_T ) - F(w_*) \right]
    \leq  \frac{1}{ n^{1/3} T^{2/3}}\left( \frac{\E \left[\|\tilde{w}_{0} - w_*\|^2 \right]}{2 (1-\beta)\gamma }  +\frac{4L\sigma^2 (1+ \beta)\gamma^2 }{6  (1-\beta)}    \right).
\end{align}
Consequently, the convergence rate of Algorithm~\ref{sgd_momentum_shuffling_01} is $\mathcal{O}\Big( \frac{1}{n^{1/3}T^{2/3}} \Big)$.
\end{cor}

\begin{proof} 
Since the learning rate $\eta_t$ is constant, we can choose $\alpha = 1$ and $K = 1+ \beta$ in Theorem \ref{thm_smg_convex}. 
From \eqref{eq_thm_smg_convex} of Theorem \ref{thm_smg_convex}, we obtain 
\begin{align*}
    & \E \left[ F( \hat{w}_T ) - F(w_*) \right] 
    \leq \frac{\E \left[\|\tilde{w}_{0} - w_*\|^2 \right]}{2 (1-\beta)T \eta}  +\frac{4L\sigma^2 (1+ \beta)}{6n  (1-\beta)} \eta^2 
.\end{align*}
Now, substituting $\eta = \frac{\gamma n^{1/3}}{T^{1/3}}$ into the last inequality, we get
\begin{align*}
    \E \left[ F( \hat{w}_T ) - F(w_*) \right]
    &\leq \frac{\E \left[\|\tilde{w}_{0} - w_*\|^2 \right]}{2 (1-\beta)\gamma n^{1/3} T^{2/3}}  +\frac{4L\sigma^2 (1+ \beta)}{6n  (1-\beta)} \frac{\gamma^2 n^{2/3}}{T^{2/3}} \nonumber\\
    &\leq \frac{1}{ n^{1/3} T^{2/3}}\left( \frac{\E \left[\|\tilde{w}_{0} - w_*\|^2 \right]}{2 (1-\beta)\gamma }  +\frac{4L\sigma^2 (1+ \beta)\gamma^2 }{6  (1-\beta)}    \right),
\end{align*}
which proves \eqref{eq:convergence_rate1}.
\end{proof}

Next, we consider an exponential scheduled learning rate. For a given epoch budget $T \geq 1$, and two positive constants $\gamma > 0$ and $\rho > 0$, we consider the following exponential learning rate, see, e.g. \citep{li2020exponential}:
\begin{align}\label{exponential_learning_rate_01} 
    \eta_t :=  \eta_0 \alpha^t , \quad \text{where}\ \alpha := \rho^{1/T} \in (0, 1).
\end{align}
Then, the following corollary states a convergence rate of Algorithm~\ref{sgd_momentum_shuffling_01} using this learning rate without any additional assumption.

\begin{cor}[\textbf{Exponential scheduled learning rate}]\label{co:exponential_LR}
Let us fix the number of epochs $T \geq 1$, and  choose a diminishing learning rate $\eta_t := \eta_0 \alpha^t =  \frac{\gamma n^{1/3}}{T^{1/3}} \alpha^t $ for $\alpha := \rho^{1/T} \in (0, 1)$ and some $\gamma > 0$ such that $\frac{\gamma \alpha n^{1/3}}{T^{1/3}} \leq \frac{1}{2L\sqrt{1+ \alpha \beta}}$ in Algorithm~\ref{sgd_momentum_shuffling_01}.
Then, under the conditions of Theorem~\ref{thm_smg_convex}, we have
\begin{align}\label{eq:convergence_rate2} 
    \E \left[ F( \hat{w}_T ) - F(w_*) \right] 
    \leq  \frac{1}{ n^{1/3} T^{2/3}}\left( \frac{\E \left[\|\tilde{w}_{0} - w_*\|^2 \right]}{2 (1-\beta) \gamma \rho }  +\frac{4L\sigma^2 (1+ \alpha \beta)\gamma^2}{6  (1-\beta)\rho} \right).
\end{align}
Consequently, the convergence rate of Algorithm~\ref{sgd_momentum_shuffling_01} is {$\mathcal{O}\Big( \frac{1}{ n^{1/3}T^{2/3}} \Big)$}.
\end{cor}

\begin{proof}

We observe that $\eta_t = \alpha \eta_{t-1}$ and $\eta_t \leq \eta_1 \leq \frac{1}{2L\sqrt{1+ \alpha \beta}} = \frac{1}{2L\sqrt{K}}$. 
Utilizing \eqref{eq_thm_smg_convex} from Theorem \ref{thm_smg_convex}, we have
\begin{align*}
    \E \left[ F( \hat{w}_T ) - F(w_*) \right] &\leq \frac{\E \left[\|\tilde{w}_{0} - w_*\|^2 \right]}{2 (1-\beta)\sum_{t=1}^T \eta_0 \alpha^t}  +\frac{4L\sigma^2 (1+ \alpha \beta)}{6n  (1-\beta)}   \frac{\sum_{t=1}^T \eta_0^3 \alpha^{3t}  }{\sum_{t=1}^T \eta_0 \alpha^t} \nonumber\\
    &\leq \frac{\E \left[\|\tilde{w}_{0} - w_*\|^2 \right]}{2 (1-\beta)\sum_{t=1}^T \eta_0 \alpha^T}  +\frac{4L\sigma^2 (1+ \alpha \beta)}{6n  (1-\beta)}   \frac{\sum_{t=1}^T \eta_0^3 \alpha^{3}  }{\sum_{t=1}^T \eta_0 \alpha^T} \nonumber\\
    &\leq \frac{\E \left[\|\tilde{w}_{0} - w_*\|^2 \right]}{2 (1-\beta)\rho T\eta_0}  +\frac{4L\sigma^2 (1+ \alpha \beta)}{6n  (1-\beta)\rho}  \eta_0^2 
.\end{align*}
Now, let $\eta_0 =  \frac{\gamma n^{1/3}}{T^{1/3}}$, the last inequality leads to
\begin{align*}
    \E \left[ F( \hat{w}_T ) - F(w_*) \right] 
    &\leq \frac{\E \left[\|\tilde{w}_{0} - w_*\|^2 \right]}{2 (1-\beta)\rho T^{2/3}\gamma n^{1/3}}  +\frac{4L\sigma^2 (1+ \alpha \beta)}{6n  (1-\beta)\rho} \frac{\gamma^2 n^{2/3}}{T^{2/3}} \nonumber\\
    &\leq  \frac{1}{ n^{1/3} T^{2/3}}\left( \frac{\E \left[\|\tilde{w}_{0} - w_*\|^2 \right]}{2 (1-\beta) \gamma \rho }  +\frac{4L\sigma^2 (1+ \alpha \beta)\gamma^2}{6  (1-\beta)\rho}   \right),
\end{align*}
which proves \eqref{eq:convergence_rate2}.
\end{proof}

\begin{rem}[Convergence Rates]
With a randomized reshuffling strategy, the convergence rate of Algorithm~\ref{sgd_momentum_shuffling_01} in a smooth and convex setting remains $\mathcal{O}( n^{-1/3}T^{-2/3})$ for the two common learning rates. 
This rate matches the state-of-the-art rate for shuffling SGD methods in a general convex setting as in \cite{mishchenko2020random,nguyen2020unified}. 
If we do not reshuffle the permutations for every epoch, then the rate for an arbitrary permutation is $\mathcal{O}(T^{-2/3})$, which is worse than the randomized reshuffling scheme with a factor of $n^{1/3}$. 
Since the analysis for unified shuffling schemes is similar and straightforward compared to our randomized analysis of Theorem \ref{thm_smg_convex}, we skip these proofs here.

\end{rem}
\begin{rem}[Learning Rate Schedules]
Our convergence results can be applied to other learning rates schedule, including diminishing learning rate (e.g. step sizes diminishes with some order of $t$) \cite{nguyen2020unified} and cosine scheduled learning rates \cite{loshchilov10sgdr,li2020exponential}. Similar to the nonconvex setting \cite{smg_tran21b}, we typically need that all the step size are lower bounded and upper bounded, following some specific order of $T$ (or $t$), which these bounds can be obtained from the design of the learning rate schedule. For convex settings, we will need an additional mild condition i.e. $\eta_t \leq \alpha \eta_{t-1}$ for some $ \alpha > 0$. While these applications are possible e.g., by considering the factor $\alpha$, in this manuscript we omit such long technical details.
\end{rem}

\subsection{Strongly Convex Case: Main Result and Consequences}
In this subsection, we present an analysis for the case when the objective function is strongly convex and each component function is convex.
\begin{thm}
\label{thm_smg_st_convex}

Suppose that Assumptions \ref{as:A1}, \ref{ass_convex}, and \ref{ass_stronglyconvex} hold for \eqref{ERM_problem_01}. 
Suppose further that a randomized reshuffling strategy is used  in Algorithm \ref{sgd_momentum_shuffling_01}. 
Let $\{w_i^{(t)}\}_{t=1}^{T}$ be generated by Algorithm \ref{sgd_momentum_shuffling_01} with a fixed momentum parameter $0\leq \beta < 1$ and an epoch learning rate $\eta_i^{(t)} := \frac{\eta_t}{n}$ for $t \geq 1$. 
Assume that $\eta_t \leq \alpha \eta_{t-1}$ for some $ \alpha > 0$
and $0 < \eta_t \leq \frac{1}{2L\sqrt{K}}$ for $t \geq 1$, where $K = 1 + \alpha\beta (1 +\mu \max_t \eta_t)$. 
Then, it holds that
\begin{align} \label{eq_thm_smg_st_convex}
    \E \left[\left\|\tilde{w}_{T}- w_*\right\|^2  \right] 
    &\leq  \frac{\E \left[\left\|\tilde{w}_{0}- w_*\right\|^2 \right]}{(1-\beta) \prod_{t=1}^{T} (1+\mu\eta_t)}
     +   \frac{4L\sigma^2}{3n}  \sum_{j=1}^T  \eta_j^3\frac{ K-\beta }{\prod_{t=j}^{T} (1+\mu\eta_t)}.
\end{align}

\end{thm}

The proof of Theorem~\ref{thm_smg_st_convex} relies on the key bound \eqref{eq:key_bound1} of Lemma \ref{lem_smg_convex} and the fact that $F( \tilde{w}_{t}) - F(w_*) \geq \frac{\mu}{2} \|\tilde{w}_{t} - w_* \|^2$  due to the $\mu$-strong convexity of $F$. 
By unrolling the recursion bound on the term $\E \left[\left\|\tilde{w}_{t}- w_*\right\|^2  \right] $, we get \eqref{eq_thm_smg_st_convex}.
The proof of Theorem~\ref{thm_smg_st_convex} is in the appendix. 
Next, we derive a convergence rate of Algorithm~\ref{sgd_momentum_shuffling_01} for the two common learning rates.
Corollary \ref{co:constant_LR_2} considers the most simple case of constant step sizes.

\begin{cor}[\textbf{Constant learning rate}]\label{co:constant_LR_2}
Let us fix the number of epochs $T \geq 1$, and  choose a constant learning rate $\eta_t := \eta =\frac{\gamma \log(n^{1/2} T)}{T} $ for some $\gamma > 0$ such that $\frac{\gamma \log(n^{1/2} T)}{T} \leq \frac{1}{2L\sqrt{K}}$ for $K = 1 + \beta (1 +\mu \eta)$ and $t\geq 1$ in Algorithm~\ref{sgd_momentum_shuffling_01}.
Then, under the conditions of Theorem~\ref{thm_smg_st_convex}, we have
\begin{align*}
\hspace{-1ex}
    \E \left[\left\|\tilde{w}_{T}- w_*\right\|^2  \right] 
    &\leq  \frac{1}{nT^2} \left(\frac{\E \big[\left\|\tilde{w}_{0}- w_*\right\|^2 \big]}{(1-\beta)} e^{- {\mu \gamma}/{2}}
    +\frac{4L\sigma^2 \gamma^2 (\log(n^{1/2} T))^2 \left( 1 +  \beta \right)}{3\mu} \right).
\end{align*}
Consequently, the convergence rate of Algorithm~\ref{sgd_momentum_shuffling_01} is $\tilde{\Ocal}(\frac{1}{nT^2})$.
\end{cor}

\begin{proof}
Since the learning rate is constant, we can choose $\alpha = 1$ and $K = 1 + \beta (1 +\mu \eta)$.
From \eqref{eq_thm_smg_st_convex} of Theorem \ref{thm_smg_st_convex} we have
\begin{align*}
    &\quad\E \left[\left\|\tilde{w}_{T}- w_*\right\|^2  \right] \leq  \frac{\E \left[\left\|\tilde{w}_{0}- w_*\right\|^2 \right]}{(1-\beta) (1+\mu\eta)^T}
     +   \frac{4L\sigma^2}{3n}  \sum_{j=1}^T  \eta^3\frac{ \left( 1-\beta +  \beta (1 + \mu \eta )\right)}{ (1+\mu\eta)^{T-j+1}}\nonumber\\
     &\leq  \frac{\E \left[\left\|\tilde{w}_{0}- w_*\right\|^2 \right]}{(1-\beta) } \left(1- \frac{\mu \eta}{2}\right)^T
     +   \frac{4L\sigma^2}{3n} \eta^3  \left( 1 +  \beta  \mu \eta \right)\sum_{j=1}^T  \frac{1}{ (1+\mu\eta)^{T-j+1}}\nonumber\\
     &\leq  \frac{\E \left[\left\|\tilde{w}_{0}- w_*\right\|^2 \right]}{(1-\beta)} \exp\left(- \frac{\mu \eta}{2}\right)^T
     +   \frac{4L\sigma^2}{3n} \eta^3  \frac{\left( 1 +  \beta  \mu \eta \right)}{ \mu\eta}\nonumber\\
     &\leq  \frac{\E \left[\left\|\tilde{w}_{0}- w_*\right\|^2 \right]}{(1-\beta)} \exp\left(- \frac{T\mu \eta}{2}\right)
     +   \frac{4L\sigma^2}{3n\mu} \eta^2 \left( 1 +  \beta\right),
\end{align*}
where we use the fact that $\mu \eta \leq 1$, $\frac{1}{1+ \mu \eta} \leq 1- \frac{\mu \eta}{2}$, and $1-x \leq \exp(-x)$. 
Now, applying the choice $\eta =\frac{\gamma \log(n^{1/2} T)}{T} $ into the above estimate, we can deduce 
\begin{align*}
    &\E \left[\left\|\tilde{w}_{T}- w_*\right\|^2  \right] \nonumber\\
    &\leq  \frac{\E \left[\left\|\tilde{w}_{0}- w_*\right\|^2 \right]}{(1-\beta)} \exp\left(- \frac{\mu \gamma \log(n^{1/2} T)}{2}\right)
     +   \frac{4L\sigma^2}{3nT^2\mu} \gamma^2 (\log(n^{1/2} T))^2 \left( 1 +  \beta \right)
     \nonumber\\
    &\leq  \frac{1}{nT^2} \left(\frac{\E \left[\left\|\tilde{w}_{0}- w_*\right\|^2 \right]}{(1-\beta)} e^{- {\mu \gamma}/{2}}
    +\frac{4L\sigma^2 \gamma^2 (\log(n^{1/2} T))^2 \left( 1 +  \beta \right)}{3\mu} \right),
\end{align*}
which proves our desired result.
\end{proof}

In the final Corollary \ref{co:exponential_LR_2}, we consider the exponential step size  \eqref{exponential_learning_rate_01}, which is described in the previous section. Our analysis is flexible that it allows various learning rate schedules, which attains the same convergence guarantees as the standard constant step sizes. 
\begin{cor}[\textbf{Exponential scheduled learning rate}]\label{co:exponential_LR_2}
Let us fix the number of epochs $T \geq 1$, and  choose a diminishing learning rate $\eta_t := \eta_0 \alpha^t =  \frac{\gamma \log(n^{1/2} T)}{T} \alpha^t $ for $\alpha := \rho^{1/T} \in (0, 1)$ and some $\gamma > 0$ such that $\frac{\gamma \alpha\log(n^{1/2} T)}{T}  \leq \frac{1}{2L\sqrt{K}}$ in Algorithm~\ref{sgd_momentum_shuffling_01}.
Then, under the conditions of Theorem~\ref{thm_smg_st_convex}, $\E \left[\left\|\tilde{w}_{T}- w_*\right\|^2  \right]$ is upper bounded by
\begin{align}\label{eq:convergence_result4}
    \frac{1}{nT^2} \left(\frac{\E \left[\left\|\tilde{w}_{0}- w_*\right\|^2 \right]}{(1-\beta)} e^{- {\mu \gamma \rho}/{2}}
    +   \frac{4L\sigma^2 (1 + 2\alpha\beta) \rho^2 \gamma^2 (\log(n^{1/2} T))^2 }{3 \mu }  \right).
\end{align}
Consequently, the convergence rate of Algorithm~\ref{sgd_momentum_shuffling_01} is $\tilde{\Ocal}(\frac{1}{nT^2})$.
\end{cor}

\begin{proof}

First, observe that $K = 1 + \alpha\beta (1 +\mu \max_t \eta_t) = 1 + \alpha\beta (1 +\mu \eta_1  )$, $\eta_t = \alpha \eta_{t-1}$ and $\eta_t \leq \eta_1 \leq \frac{1}{2L\sqrt{K}}$. 
The learning rate is also lower bounded: $\eta_t \geq \eta_T = \eta_0 \alpha^T = \eta_0 \rho$. 
Note that $\mu \eta_t \leq 1$, then $\frac{1}{1+ \mu \eta_t} \leq 1- \frac{\mu \eta_t}{2}$ and $1-x \leq \exp(-x)$.
\\    
Next, from Theorem \ref{thm_smg_st_convex}, we have
\begin{align*} 
    &\quad \E \left[\left\|\tilde{w}_{T}- w_*\right\|^2  \right] \\
    &\leq  \frac{\E \left[\left\|\tilde{w}_{0}- w_*\right\|^2 \right]}{(1-\beta) \prod_{t=1}^{T} (1+\mu\eta_t)}
     +   \frac{4L\sigma^2}{3n}  \sum_{j=1}^T  \eta_0^3 \alpha^{3j} \frac{ K-\beta }{\prod_{t=j}^{T} (1+\mu\eta_t)}\nonumber\\
     &\leq  \frac{\E \left[\left\|\tilde{w}_{0}- w_*\right\|^2 \right]}{(1-\beta)  (1+\mu\eta_0 \rho)^T } 
     +   \frac{4L\sigma^2 (K-\beta)}{3n}  \sum_{j=1}^T  \eta_0^3 \alpha^{3j} \frac{ 1 }{\prod_{t=j}^{T} (1+\mu\eta_0 \rho)}
     \nonumber\\
     &\leq  \frac{\E \left[\left\|\tilde{w}_{0}- w_*\right\|^2 \right]}{(1-\beta) } \left(1- \frac{\mu \eta_0 \rho}{2}\right)^T 
     +   \frac{4L\sigma^2 (K-\beta)}{3n}  \sum_{j=1}^T  \eta_0^3 \alpha^{3T} \frac{ 1 }{ (1+\mu\eta_0 \rho)^{T-j+1}} \nonumber\\
     &\leq  \frac{\E \left[\left\|\tilde{w}_{0}- w_*\right\|^2 \right]}{(1-\beta) }  \exp\left(- \frac{\mu \eta_0 \rho}{2}\right)^T
     +   \frac{4L\sigma^2 (K-\beta)}{3n} \eta_0^3 \alpha^{3T}  \sum_{j=1}^T   \frac{ 1 }{ (1+\mu\eta_0 \rho)^{T-j+1}} \nonumber\\
     &\leq  \frac{\E \left[\left\|\tilde{w}_{0}- w_*\right\|^2 \right]}{(1-\beta) }  \exp\left(- \frac{\mu T\eta_0 \rho}{2}\right)
     +   \frac{4L\sigma^2 (K-\beta) \rho^3 \eta_0^3}{3n } \frac{1}{\mu\eta_0 \rho}\nonumber\\
     &\leq  \frac{\E \left[\left\|\tilde{w}_{0}- w_*\right\|^2 \right]}{(1-\beta) }  \exp\left(- \frac{\mu T\eta_0 \rho}{2}\right)
     +   \frac{4L\sigma^2 (K-\beta) \rho^2 \eta_0^2}{3n \mu}. 
\end{align*}
Finally, applying the choice $\eta_0 =\frac{\gamma \log(n^{1/2} T)}{T} $ we have 
\begin{align*}
    &\E \left[\left\|\tilde{w}_{T}- w_*\right\|^2  \right] \nonumber\\
     &\leq  \frac{\E \left[\left\|\tilde{w}_{0}- w_*\right\|^2 \right]}{(1-\beta) }  \exp\left(- \frac{\mu \gamma \log(n^{1/2} T) \rho}{2}\right)
     +   \frac{4L\sigma^2 (K-\beta) \rho^2 }{3n \mu } \frac{\gamma^2 (\log(n^{1/2} T))^2}{T^2} \nonumber\\
    &\leq  \frac{1}{nT^2} \left(\frac{\E \left[\left\|\tilde{w}_{0}- w_*\right\|^2 \right]}{(1-\beta)} e^{- {\mu \gamma \rho}/{2}}
    +   \frac{4L\sigma^2 (1 + 2\alpha\beta) \rho^2 \gamma^2 (\log(n^{1/2} T))^2 }{3 \mu }  \right),
\end{align*}
which proves \eqref{eq:convergence_result4}.
\end{proof}

\begin{rem}
With a randomized reshuffled permutation strategy, the convergence rate of Algorithm~\ref{sgd_momentum_shuffling_01} for a strongly convex setting is $\tilde{\Ocal}(\frac{1}{nT^2})$ for two different step sizes. 
This rate matches  the state-of-the-art rates for shuffling SGD methods in a strongly convex setting, see, e.g., \cite{mishchenko2020random,nguyen2020unified}. 
To our best knowledge, this is the first work for a shuffling momentum method achieving SOTA convergence rate in a strongly convex setting.
Similar to the convex case, we can derive an analysis for unified data shuffling schemes and  prove an $\tilde{\Ocal}(\frac{1}{T^2})$ rate 
(which is worse than the randomized reshuffling scheme with a factor of $n$) using similar arguments as in our analysis of Theorem \ref{thm_smg_st_convex}.
\end{rem}

Similar to the convex case, we note that other step size strategies \cite{nguyen2020unified,loshchilov10sgdr,li2020exponential} can be adapted to our strongly convex analysis, by considering the factor $\alpha$. However, we omit such discussion due to long technical details. 
\section{Numerical Experiments}\label{sec:experiments}
We demonstrate the theoretical results developed in the previous sections.
We conduct a small numerical experiment on convex problems on a logistic regression model using real data sets and compare the performance of our SMG algorithm with some state-of-the-art stochastic gradient methods in the literature.

\textbf{Experiment settings. }
We consider the following convex problem arising from binary classification:
\begin{align*}
   \min_{w \in \mathbb{R}^d} \Big \{ F(w) \! := \! \frac{1}{n} \sum_{i=1}^n \log(1 \! + \! \exp(- y_i x_i^\top w )) \!  \Big \}, 
\end{align*}
where $\{(x_i, y_i)\}_{i=1}^n$ is a given set of training samples. 
We perform experiments on two real datasets from \texttt{LIBSVM} \citep{LIBSVM}, namely \texttt{w8a} with $49,749$ samples and \texttt{ijcnn1} with  $91,701$ samples.
The experiment was repeated three times with different random seeds and we report the average metrics (training loss and test accuracy) of these results.

We compare our SMG and a standard stochastic gradient scheme (SGD) and two other state-of-the-art methods: stochastic gradient  with momentum (SGD-M) \citep{polyak1964some} and Adam \citep{Kingma2014}. 
To obtain a fair comparison, we employ a randomized reshuffling scheme across all methods. 
For the latter two algorithms, we adopt the hyper-parameter settings that are recommended and widely utilized in practical applications. 
For  SMG, we set the parameter $\beta := 0.5$ as it is the recommended parameter in our paper \cite{smg_tran21b}. 
 
We fine-tune each algorithm with a grid search strategy for a constant step size and report the best result we obtain. 
For SGD, SGM, and SGD-M, the searching grid is $\{0.1, 0.01, 0.001\}$. 
We use the following rule to update the weights for SGD-M algorithm: 
\begin{align*}
    &m_{i+1}^{(t)} := \beta m_{i}^{(t)} + g_i^{(t)},\nonumber\\
    &w_{i+1}^{(t)} := w_{i}^{(t)} - \eta_{i}^{(t)} m_{i+1}^{(t)},
\end{align*}
where $g_i^{(t)}$ is the $(i+1)$-th gradient at epoch $t$. 
We choose this momentum update since it is one of the standard updates and it is implemented in PyTorch with the default value $\beta = 0.9$. 
Since the default learning rate for Adam is $0.001$, we let our searching grid for Adam be $\{0.01, 0.001, 0.0001\}$. 
The two momentum hyper-parameters for Adam is $\beta_1 := 0.9$ and $\beta_2 := 0.999$.


\begin{figure}[!ht] 
\begin{center}
\includegraphics[width=0.85\textwidth]{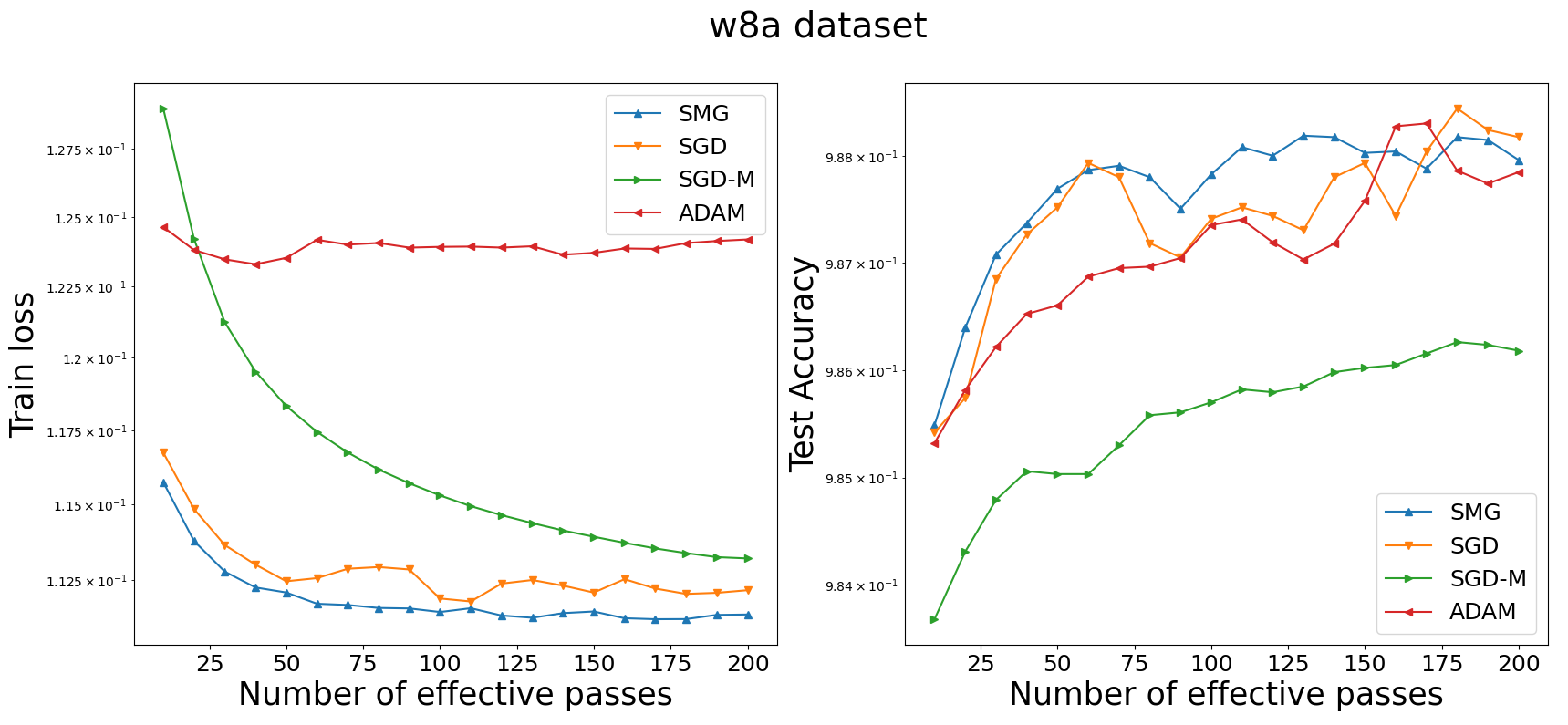}
\includegraphics[width=0.85\textwidth]{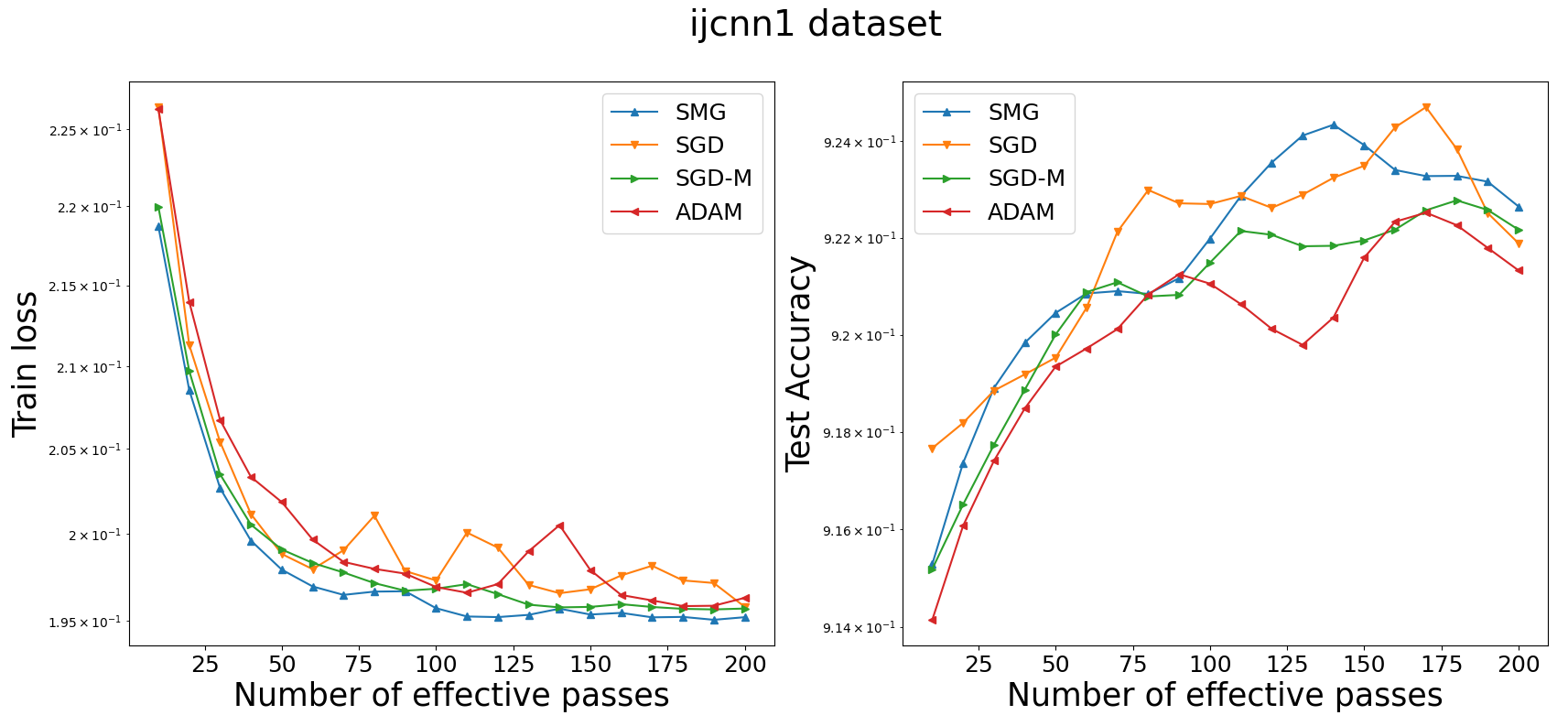}
\caption{The train loss (left) and test accuracy (right) produced by SMG, SGD, SGD-M, and Adam for the \texttt{w8a} and \texttt{ijcnn1} datasets, respectively.}
\label{fig}
\end{center}
\end{figure}
\textbf{Numerical results. }
We report the results of our experiments for the two datasets  in Figure \ref{fig}.
For the \texttt{w8a} dataset, SMG demonstrates notably superior train loss compared to Adam and SGD-M, and slightly outperforms SGD. 
For the \texttt{ijcnn1} dataset, SMG exhibits lower train losses compared to SGD and Adam, and has  comparable performance to SGD with momentum.
On the test accuracy, although there is no clear dominance among the methods, SMG consistently performs well throughout the training process and is comparable to the best method in our test.

Note that problem we aim to solve is minimizing the training loss, which is not always the same as generalizing well on the testing data. We tune the learning rates so that each methods has the best training performance (without using validation data). Hence, the test accuracy in our experiment only served as some reference as it is comparable with other algorithms and it does not necessary reflect the generalization performance of our method.

\section{Conclusions}\label{sec:concl}
This paper studies the Shuffling Momentum Gradient (SMG) algorithm introduced in \cite{smg_tran21b}, a variant of the Shuffling SGD algorithm augmented with an anchor momentum. 
While our earlier work extensively analyzed SMG in non-convex settings, this study extends the investigation to two cases: convex and strongly convex settings. Our analysis is new and matches the current state-of-the-art convergence rate of the shuffling gradient-type methods. This work and \cite{smg_tran21b} provide a complete analysis of the Shuffling Momentum Gradient algorithm for both convex and nonconvex problems. 

Several promising research directions arises from this work. For instance, first, investigating other momentum methods and adaptive step sizes is an interesting problem for further exploration. Second, it is desirable to study our approach to  address composite problems and minimax optimization. Third, while the study of most shuffling methods focus on the complexity per epoch, the convergence analysis for each iteration may offer additional understanding and insights of the algorithm's behavior.

\begin{appendices}
\section{Technical Notations and Expressions}\label{sec:pre_results}
First, we introduce some additional notations for our analysis in the sequel.
For $t \geq 1$ and  $i = 0,\dots, n$, we denote
\begin{align}
    & A_i^{(t)} := \Big\Vert  \sum_{j=0}^{i-1}  g_{j}^{(t)} \Big\Vert ^2   \quad  \text{and} \quad
     B_i^{(t)} := \Big\Vert \sum_{j = i}^{n-1} g_{j}^{(t)}\Big\Vert^2.     \label{define_AB}
\end{align}
We adopt the convention $\sum_{j=0}^{-1}  g_{j}^{(t)} = 0$, and $\sum_{j=n}^{n-1}  g_{j}^{(t)} = 0$ in the above definitions.
Next, we collect the following expressions from Algorithm~\ref{sgd_momentum_shuffling_01} to use later in our analysis.
From the update rules of Algorithm~\ref{sgd_momentum_shuffling_01}, for $t > 1$, we have
\begin{align}\label{update_m_0^{T}}
\hspace{-2.5ex}
    m_0^{(t)} = \tilde{m}_{t-1} = v_{n}^{(t-1)} \overset{\eqref{eq:main_update}}{=} v_{n-1}^{(t-1)} + \frac{1}{n} g_{n-1}^{(t-1)}
    = v_{0}^{(t-1)} + \frac{1}{n} \sum_{j=0}^{n-1} g_{j}^{(t-1)} = \frac{1}{n} \sum_{j=0}^{n-1} g_{j}^{(t-1)}.
\hspace{-1ex}    
\end{align}
If $t = 1$, then we set $m_0^{(1)} = \tilde{m}_{0} = \textbf{0} \in\R^d$.
From the update $w_{i+1}^{(t)} := w_{i}^{(t)} - \eta_i^{(t)} m_{i+1}^{(t)}$  with $\eta_i^{(t)} := \frac{\eta_t}{n}$, for $i = 1, 2,\dots ,n-1$, at Step \ref{alg:A1_step4} of Algorithm~\ref{sgd_momentum_shuffling_01}, we can derive
\begin{align}
    w_{i}^{(t)} &\overset{\eqref{eq:main_update}}{=} w_{i-1}^{(t)} - \frac{\eta_t}{n} m_{i}^{(t)} = w_{0}^{(t)} - \frac{\eta_t}{n} \sum_{j=0}^{i-1} m_{j+1}^{(t)}, \quad  t \geq 1.   \label{update_w_i^{T}}
\end{align}
Note that $\sum_{j=0}^{-1} m_{j+1}^{(t)}$ and $\sum_{j = n}^{n-1} m_{j+1}^{(t-1)} = 0$ by convention, these equations also hold true for $i=0$ and $i=n$.

Now, letting $i=n$ in Equation \eqref{update_w_i^{T}}, for all $t\geq 1$, we have
\begin{align*}
    w_{n}^{(t)} - w_{0}^{(t)} &= - \frac{\eta_t}{n} \sum_{j=0}^{n-1} m_{j+1}^{(t)} 
    \overset{\eqref{eq:main_update}}{=}  - \frac{\eta_t}{n} \sum_{j=0}^{n-1} \big( \beta m_{0}^{(t)} + (1-\beta) g_{j}^{(t)}\big) \nonumber\\
    &=  - \frac{\eta_t}{n} \Big( n \beta m_{0}^{(t)} +  (1-\beta)  \sum_{j=0}^{n-1}  g_{j}^{(t)}  \Big).
\end{align*}
Since $m_0^{(t)} = \frac{1}{n} \sum_{j=0}^{n-1} g_{j}^{(t-1)}$ for $t \geq 2$ (due to \eqref{update_m_0^{T}}), we have the following update
\begin{align}
    w_{n}^{(t)} - w_{0}^{(t)} &= - \frac{\eta_t}{n}  \sum_{j=0}^{n-1} \Big( \beta g_{j}^{(t-1)} + (1-\beta)   g_{j}^{(t)} \Big).
      \label{update_epoch_02}
\end{align}
For $t = 1$, since $m_{0}^{(t)} =\mathbf{0}$, we get
\begin{align}
    w_{n}^{(t)} - w_{0}^{(t)} =- \frac{\eta_t}{n}  (1-\beta)  \sum_{j=0}^{n-1}  g_{j}^{(t)}.\label{update_epoch_03}
\end{align}

\section{Technical Lemmas}
This appendix provides necessary technical lemmas used in our analysis.

\subsection{Lemma \ref{lem_key_rr}: Sampling without replacement}
We need \citep{mishchenko2020random}[Lemma 1]  for sampling without replacement in our analysis.
For complete references, we recall it here.

\begin{lem}[Lemma 1 in \cite{mishchenko2020random}]\label{lem_key_rr}
Let $X_1, \cdots, X_n \in \R^d$ be fixed vectors, $\bar{X} := \frac{1}{n} \sum_{i=1}^n X_i$ be their average and $\sigma^2 := \frac{1}{n} \sum_{i=1}^n \|X_i -\bar{X}\|^2$ 
be the population variance. Fix any $k \in \{1,\cdots, n\}$, let $X_{\pi_1}, \cdots, X_{\pi_k}$ be sampled uniformly without replacement from $\{X_1, \cdots, X_n\}$ and $\bar{X}_\pi$ be their average. 
Then, the sample average and the variance are given, respectively by
\begin{align*}
    \E [\bar{X}_\pi] = \bar{X} \qquad \text{and} \qquad \E \left[ \|\bar{X}_\pi - \bar{X}\|^2 \right] = \frac{n-k}{k(n-1)} \sigma^2.
\end{align*}
\end{lem}

Since $\pi^{(t)} = (\pi^{(t)}(1),\dots,\pi^{(t)}(n))$ is uniformly  sampled  at  random  without  replacement  from $\{1,\dots,n\}$, by Lemma~\ref{lem_key_rr}, we have $\mathbb{E}\left[ \frac{1}{n-i} \sum_{j=i}^{n-1} \nabla f ( w_* ; \pi^{(t)} (j + 1)) \right] = \nabla F ( w_* )$ and
\begin{align}
    &\quad\mathbb{E} \left[ \left\| \sum_{j=i}^{n-1} \nabla f ( w_* ; \pi^{(t)} (j + 1))  \right\|^2 \right] \nonumber\\
    &= (n-i)^2 \mathbb{E} \left[ \left\| \frac{1}{n-i} \sum_{j=i}^{n-1} \nabla f ( w_* ; \pi^{(t)} (j + 1))  \right\|^2 \right] \nonumber\\
    &= (n-i)^2 \mathbb{E} \left[ \left\| \frac{1}{n-i} \sum_{j=i}^{n-1} \nabla f ( w_* ; \pi^{(t)} (j + 1))  - \nabla F ( w_* )  \right\|^2 \right] \nonumber\\
    &= \frac{(n-i)^2 i}{(n-i)(n-1)} \frac{1}{n} \sum_{j=0}^{n-1} \left\| \nabla f ( w_* ; \pi^{(t)} (j + 1))  \right\|^2 \nonumber\\
    &\overset{\eqref{defn_finite}}{\leq} \frac{(n-i) i}{(n-1)} \sigma^2.  \label{upper_bound_stc_exp}
\end{align}

\subsection{Bound expected squared norm of sum of gradients}
\begin{lem}\label{lem_bound_Bit}
Let  $w_*$ be an optimal solution of $F$, and $\{w_i^{(t)}\}$ be generated by Algorithm \ref{sgd_momentum_shuffling_01} with $0\leq \beta < 1$ and $\eta_i^{(t)} := \frac{\eta_t}{n}$ for every $t \geq 1$.
Assume that Assumption~\ref{as:A1} and Assumption~\ref{ass_convex} hold, {with the application of a randomized reshuffling strategy}, we have the following bounds for $t \geq 1$:
\begin{align*}
    \E\left[B_0^{(t)}\right] \leq 4nL \cdot \E[C_t]~~~~~\text{and}~~~~~    \sum_{i=0}^{n-1} \E\left[B_i^{(t)}\right] 
   \leq 4n^2 L\cdot \E[C_t]  +  \frac{2n^2\sigma^2}{3}.
\end{align*}
\end{lem}
\begin{proof}
We start with the definition of $B_i^{(t)}$: 
\begin{align*}
    B_i^{(t)} &= \left \|   \sum_{j=i}^{n-1}  g_{j}^{(t)}  \right \|^2 \nonumber\\
    &= \left \|   \sum_{j=i}^{n-1}  \left(g_{j}^{(t)} - \nabla f(w_*; \pi^{(t)}(j+1) ) \right) + \sum_{j=i}^{n-1}  \nabla f(w_*; \pi^{(t)}(j+1) ) \right\|^2 \nonumber\\
    &\overset{(a)}{\leq}  2\left \|   \sum_{j=i}^{n-1}  \left(g_{j}^{(t)} - \nabla f(w_*; \pi^{(t)}(j+1) ) \right) \right\|^2 + 2\left \| \sum_{j=i}^{n-1}  \nabla f(w_*; \pi^{(t)}(j+1) ) \right\|^2 \nonumber\\
    &\overset{(a)}{\leq}  2 (n-i) \sum_{j=i}^{n-1} \left \| g_{j}^{(t)} - \nabla f(w_*; \pi^{(t)}(j+1) ) \right\|^2 + 2\left \| \sum_{j=i}^{n-1}  \nabla f(w_*; \pi^{(t)}(j+1) ) \right\|^2\nonumber\\
    &\overset{(b)}{\leq} 2 (n-i) \cdot 2L \sum_{j=i}^{n-1} D_{j}^{(t)} (w_*; w_{j}^{(t)}) +  2\left \| \sum_{j=i}^{n-1}  \nabla f(w_*; \pi^{(t)}(j+1) ) \right\|^2,
\end{align*}
where  $(a)$ is from from the Cauchy-Schwarz inequality, and the last line  $(b)$ follows by the inequality \eqref{eq_div_1}.  
Now taking expectation to both sides and using \eqref{upper_bound_stc_exp}, we have
\begin{align*}
    \E\left[B_i^{(t)}\right] &= 4L (n-i) \E \left[ \sum_{j=i}^{n-1} D_{j}^{(t)} (w_*; w_{j}^{(t)})\right] +  2 \left [\E \left \| \sum_{j=i}^{n-1}  \nabla f(w_*; \pi^{(t)}(j+1) ) \right\|^2\right].\nonumber\\
    &\overset{\eqref{define_C}\eqref{upper_bound_stc_exp}}{\leq} 4L (n-i) \E[C_t] + 2\frac{(n-i) i}{(n-1)} \sigma^2.
\end{align*}
For the special case $i=0$, we have: 
\begin{align*}
    \E\left[B_0^{(t)}\right] {\leq} 4L n \E[C_t].
\end{align*}
Summing up the expression from $i := 0$ to $i := n-1$, we get 
\begin{align*}
   \sum_{i=0}^{n-1} \E\left[B_i^{(t)}\right] &\leq 4L\cdot \E[C_t] \sum_{i=0}^{n-1} (n-i)  + 2\sigma^2 \sum_{i=0}^{n-1}\frac{(n-i) i}{(n-1)}.\nonumber\\
   &\leq 4n^2 L\cdot \E[C_t]  +  \frac{2n^2\sigma^2}{3},
\end{align*}
where we use the facts that $\sum_{i=0}^{n-1} (n-i) \leq n^2$ and $\sum_{i=0}^{n-1}\frac{(n-i) i}{(n-1)} \leq \frac{n(n+1)}{6} \leq \frac{n^2}{3}$.
\end{proof}

\subsection{Bound expected sum of Bregman divergence (1)}
\begin{lem}\label{lem_bound_term_1}
Under the same setting as of Lemma~\ref{lem_bound_Bit}, it holds that
\begin{align*}
    \sum_{i=0}^{n-1} \E \left [D_{i}^{(t)}(w_{n}^{(t)}; w_{i}^{(t)}) \right]
    \leq 4 L^2\eta_t^2 F_t + \frac{2}{3}\eta_t^2  (1-\beta)L\sigma^2, \quad \text{ for } t \geq 1.
\end{align*}
\end{lem}
\begin{proof}
We start with the case $t \geq 2$:
\begin{align*}
    D_{i}^{(t)}(w_{n}^{(t)}; w_{i}^{(t)}) &\overset{\eqref{eq_div_2}}{\leq} \frac{L}{2} \left\|w_{n}^{(t)} -  w_{i}^{(t)} \right\|^2\nonumber\\
    &\leq \frac{L}{2} \frac{\eta_t^2}{n^2} \left \| \beta (n -i) m_{0}^{(t)} +  (1-\beta)  \sum_{j=i}^{n-1}  g_{j}^{(t)}  \right \|^2 \nonumber\\
    &\overset{\eqref{update_m_0^{T}}}{\leq} \frac{L}{2} \frac{\eta_t^2}{n^2} \left \| \beta \frac{n-i}{n} \sum_{i=0}^{n-1} g_{i}^{(t-1)} +  (1-\beta)  \sum_{j=i}^{n-1}  g_{j}^{(t)}  \right \|^2 \nonumber\\
    &\overset{(a)}{\leq} \frac{L}{2} \frac{\eta_t^2}{n^2} \left[  \beta\left \|  \frac{n-i}{n} \sum_{i=0}^{n-1} g_{i}^{(t-1)}  \right \|^2 + (1-\beta) \left \|   \sum_{j=i}^{n-1}  g_{j}^{(t)}  \right \|^2 \right] \nonumber\\
    & \overset{\eqref{define_AB}, \eqref{define_AB}}{=} \frac{L}{2} \frac{\eta_t^2}{n^2} \left[  \beta \frac{(n-i)^2}{n^2}  B_0^{(t-1)} + (1-\beta) B_i^{(t)}  \right],
\end{align*}
where $(a)$ follows from the convexity of $\norms{\cdot}^2$ for $0\leq \beta < 1$. 
Summing up the expression from $i := 0$ to $i := n-1$ and talking expectation, we get 
\begin{align*}
    \sum_{i=0}^{n-1} \E \left [D_{i}^{(t)}(w_{n}^{(t)}; w_{i}^{(t)}) \right]
    &\leq \frac{L}{2} \frac{\eta_t^2}{n^2} \left[  \beta \frac{\sum_{i=0}^{n-1}(n-i)^2}{n^2} \E \left [ B_0^{(t-1)}\right] + (1-\beta) \sum_{i=0}^{n-1} \E \left [B_i^{(t)} \right] \right]\nonumber\\
    &\overset{(b)}{\leq}  \frac{L}{2} \frac{\eta_t^2}{n^2} \left[  \beta \frac{n^3}{n^2} \E \left [ B_0^{(t-1)}\right] + (1-\beta) \sum_{i=0}^{n-1} \E \left [B_i^{(t)} \right] \right]\nonumber\\
    &\overset{(c)}{\leq} \frac{L}{2} \frac{\eta_t^2}{n^2} \left[  4\beta n^2L \cdot \E[C_{t-1}] + (1-\beta) \left(4n^2 L\cdot \E[C_t]  +  \frac{2n^2\sigma^2}{3}  \right)\right]\nonumber\\
    &\leq \frac{L\eta_t^2}{2} \left[ 4 \beta L \cdot \E[C_{t-1}] + 4(1-\beta) L\cdot \E[C_t]  +  (1-\beta) \frac{2\sigma^2}{3} \right]\nonumber\\
    &\overset{\eqref{define_F}}{\leq} 2 L^2\eta_t^2 F_t + \frac{L\eta_t^2}{2} (1-\beta) \frac{2\sigma^2}{3} \nonumber\\
    &\leq 4 L^2\eta_t^2 F_t + \frac{2}{3}\eta_t^2  (1-\beta)L\sigma^2,
\end{align*}
where $(b)$ follows from the fact that $\sum_{i=0}^{n-1}(n-i)^2
\leq n^3$, and in $(c)$ we use the results from Lemma \ref{lem_bound_Bit}.

We do the same for the case $t =1$:
\begin{align*}
    D_{i}^{(t)}(w_{n}^{(t)}; w_{i}^{(t)}) &\overset{\eqref{eq_div_2}}{\leq} \frac{L}{2} \left\|w_{n}^{(t)} -  w_{i}^{(t)} \right\|^2 
    \leq \frac{L}{2} \frac{\eta_t^2}{n^2} \left \| \beta (n -i) m_{0}^{(t)} +  (1-\beta)  \sum_{j=i}^{n-1}  g_{j}^{(t)}  \right \|^2 \nonumber\\
    &\leq \frac{L}{2} \frac{\eta_t^2}{n^2} \left \|  (1-\beta)  \sum_{j=i}^{n-1}  g_{j}^{(t)}  \right \|^2 
    \leq \frac{L}{2} \frac{\eta_t^2}{n^2} \left[   (1-\beta)^2 \left \|   \sum_{j=i}^{n-1}  g_{j}^{(t)}  \right \|^2 \right] \nonumber\\
    & \overset{\eqref{define_AB}}{=} \frac{L}{2} \frac{\eta_t^2}{n^2} \left[   (1-\beta) B_i^{(t)}  \right].
\end{align*}
Summing up the expression from $i := 0$ to $i := n-1$ and talking expectation, we get 
\begin{align*}
    \sum_{i=0}^{n-1} \E \left [D_{i}^{(t)}(w_{n}^{(t)}; w_{i}^{(t)}) \right]
    &\leq \frac{L}{2} \frac{\eta_t^2}{n^2} \left[   (1-\beta) \sum_{i=0}^{n-1} \E \left [B_i^{(t)} \right] \right]\nonumber\\
    &\overset{(d)}{\leq} \frac{L}{2} \frac{\eta_t^2}{n^2} \left[   (1-\beta) \left(4n^2 L\cdot \E[C_t]  +  \frac{2n^2\sigma^2}{3}  \right)\right]\nonumber\\
    &\leq \frac{L\eta_t^2}{2} \left[ 4(1-\beta) L\cdot \E[C_t]  +  (1-\beta) \frac{2\sigma^2}{3} \right]\nonumber\\
    &\overset{\eqref{define_F}}{\leq} 2 L^2\eta_t^2 F_t + \frac{L\eta_t^2}{2} (1-\beta) \frac{2\sigma^2}{3} \nonumber\\
    &\leq 4 L^2\eta_t^2 F_t + \frac{2}{3}\eta_t^2  (1-\beta)L\sigma^2,
\end{align*}
where in $(d)$ we use the results from Lemma \ref{lem_bound_Bit}.
Hence the statement of Lemma \ref{lem_bound_term_1} is true for all $t \geq 1$.
\end{proof}

\subsection{Bound expected sum of Bregman divergence (2)}
\begin{lem}\label{lem_bound_term_2}
Under the same setting as of Lemma~\ref{lem_bound_Bit}, the following holds for $t\geq 2$:
\begin{align*}
    \sum_{i=0}^{n-1}\E \left [ D_{i}^{(t-1)}(w_{n}^{(t)}; w_{i}^{(t-1)}) \right]
    &\leq 4L^2 \eta_t^2 F_t + 4L^2\eta_{t-1}^2 F_{t-1}+  \frac{2}{3} \eta_{t-1}^2(1-\beta) L\sigma^2. 
\end{align*}
\end{lem}
\begin{proof}
We start with the case $t \geq 3$: 
\begin{align*}
    D_{i}^{(t-1)}(w_{n}^{(t)}; w_{i}^{(t-1)}) 
    &\overset{\eqref{eq_div_2}}{\leq} \frac{L}{2} \left\|w_{n}^{(t)} -  w_{i}^{(t-1)} \right\|^2\nonumber\\
    &\leq L\left\|w_{n}^{(t)} - w_{n}^{(t-1)} \right\|^2 + L \left\| w_{n}^{(t-1)} - w_{i}^{(t-1)} \right\|^2,
\end{align*}
where we use the inequality $\| u + v \|^2 \leq 2 \| u \|^2 + 2 \| v \|^2$. Note that $w_{n}^{(t-1)} = w_{0}^{(t)}$, using the result 
$$\left\|w_{n}^{(t)} -  w_{i}^{(t)} \right\|^2 = \frac{\eta_t^2}{n^2} \left[  \beta \frac{(n-i)^2}{n^2}  B_0^{(t-1)} + (1-\beta) B_i^{(t)}  \right]$$
from Lemma \ref{lem_bound_term_1} for $t$ and $t-1$, we have 
\begin{align*}
\hspace{-2ex}
    &\quad D_{i}^{(t-1)}(w_{n}^{(t)}; w_{i}^{(t-1)}) \nonumber\\
    &\leq L\left\|w_{n}^{(t)} - w_{0}^{(t)}\right\|^2 + L \left\| w_{n}^{(t-1)} - w_{i}^{(t-1)} \right\|^2\nonumber\\
    &\leq L \frac{\eta_t^2}{n^2} \left[  \beta B_0^{(t-1)} + (1-\beta) B_0^{(t)}\right] + L \frac{\eta_{t-1}^2}{n^2} \left[  \beta \frac{(n-i)^2}{n^2}  B_0^{(t-2)} + (1-\beta) B_i^{(t-1)}\right].
\hspace{-2ex}
\end{align*}
Summing up the expression from $i := 0$ to $i := n-1$ and taking expectation, we get 
\begin{align*}
    &\sum_{i=0}^{n-1}\E \left [ D_{i}^{(t-1)}(w_{n}^{(t)}; w_{i}^{(t-1)}) \right]\nonumber\\
    &\leq L \frac{\eta_t^2}{n} \left[  \beta \E \left [B_0^{(t-1)}\right] + (1-\beta) \E \left [B_0^{(t)}\right]\right] \nonumber\\
    & \quad \ + L \frac{\eta_{t-1}^2}{n^2} \left[  \beta \frac{\sum_{i=0}^{n-1}(n-i)^2}{n^2}  \E \left [B_0^{(t-2)} \right]+ (1-\beta) \sum_{i=0}^{n-1}\E \left [ B_i^{(t-1)}\right]\right]\nonumber\\
    &\overset{(a)}{\leq} L \frac{\eta_t^2}{n} \left[  4\beta nL \cdot \E[C_{t-1}] + 4(1-\beta) nL \cdot \E[C_{t}]\right] \nonumber\\
    & \quad \ + L \frac{\eta_{t-1}^2}{n^2} \left[  4\beta n^2L \cdot \E[C_{t-2}]+ (1-\beta) \left[ 4n^2 L\cdot \E[C_{t-1}]  +  \frac{2n^2\sigma^2}{3}  \right]\right]\nonumber\\
    &\leq 4L^2 \eta_t^2 \left[  \beta \cdot \E[C_{t-1}] + (1-\beta)  \cdot \E[C_{t}]\right] \nonumber\\
    & \quad \ + 4L^2\eta_{t-1}^2 \left[  \beta \cdot \E[C_{t-2}]+  (1-\beta) \E[C_{t-1}]   \right]+   (1-\beta) L \frac{\eta_{t-1}^2}{n^2} \frac{2n^2\sigma^2}{3}\nonumber\\
    &\leq 4L^2 \eta_t^2 F_t + 4L^2\eta_{t-1}^2 F_{t-1}+  \frac{2}{3} \eta_{t-1}^2(1-\beta) L\sigma^2,
\end{align*}
where $(a)$ follows from the fact that $\sum_{i=0}^{n-1}(n-i)^2  \leq n^3$ and from the results of Lemma \ref{lem_bound_Bit}.
Now we analyze similarly for the case $t = 2$: 
\begin{align*}
    D_{i}^{(t-1)}(w_{n}^{(t)}; w_{i}^{(t-1)}) 
    &\overset{\eqref{eq_div_2}}{\leq} \frac{L}{2} \left\|w_{n}^{(t)} -  w_{i}^{(t-1)} \right\|^2\nonumber\\
    &\leq L\left\|w_{n}^{(t)} - w_{n}^{(t-1)} \right\|^2 + L \left\| w_{n}^{(t-1)} - w_{i}^{(t-1)} \right\|^2,
\end{align*}
where we use the inequality $\| u + v \|^2 \leq 2 \| u \|^2 + 2 \| v \|^2$. Note that $w_{n}^{(t-1)} = w_{0}^{(t)}$, using the result 
\begin{align*}
    \left\|w_{n}^{(t)} -  w_{i}^{(t)} \right\|^2 &= \frac{\eta_t^2}{n^2} \left[  \beta \frac{(n-i)^2}{n^2}  B_0^{(t-1)} + (1-\beta) B_i^{(t)}  \right] &&\text{ for } t =2,&\nonumber\\
    \left\|w_{n}^{(t)} -  w_{i}^{(t)} \right\|^2 &= \frac{\eta_t^2}{n^2} \left[ (1-\beta) B_i^{(t)}  \right] &&\text{ for } t =1,&
\end{align*}
from Lemma \ref{lem_bound_term_1}, we have 
\begin{align*}
    D_{i}^{(t-1)}(w_{n}^{(t)}; w_{i}^{(t-1)}) 
    &\leq L\left\|w_{n}^{(t)} - w_{0}^{(t)}\right\|^2 + L \left\| w_{n}^{(t-1)} - w_{i}^{(t-1)} \right\|^2\nonumber\\
    &\leq L \frac{\eta_t^2}{n^2} \left[  \beta B_0^{(t-1)} + (1-\beta) B_0^{(t)}\right] + L \frac{\eta_{t-1}^2}{n^2} \left[ (1-\beta) B_i^{(t-1)}\right].
\end{align*}
Summing up the expression from $i := 0$ to $i := n-1$ and taking expectation, we get 
\begin{align*}
    &\quad \sum_{i=0}^{n-1}\E \left [ D_{i}^{(t-1)}(w_{n}^{(t)}; w_{i}^{(t-1)}) \right]\nonumber\\
    &\leq L \frac{\eta_t^2}{n} \left[  \beta \E \left [B_0^{(t-1)}\right] + (1-\beta) \E \left [B_0^{(t)}\right]\right]  
    + L \frac{\eta_{t-1}^2}{n^2} \left[ (1-\beta) \sum_{i=0}^{n-1}\E \left [ B_i^{(t-1)}\right]\right]\nonumber\\
    &\overset{(a)}{\leq} L \frac{\eta_t^2}{n} \left[  4\beta nL \cdot \E[C_{t-1}] + 4(1-\beta) nL \cdot \E[C_{t}]\right] \nonumber\\
    & \quad \ + L \frac{\eta_{t-1}^2}{n^2} \left[ (1-\beta) \left[ 4n^2 L\cdot \E[C_{t-1}]  +  \frac{2n^2\sigma^2}{3}  \right]\right]\nonumber\\
    &\leq 4L^2 \eta_t^2 \left[  \beta \cdot \E[C_{t-1}] + (1-\beta)  \cdot \E[C_{t}]\right] \nonumber\\
    & \quad \ + 4L^2\eta_{t-1}^2 \left[  (1-\beta) \E[C_{t-1}]   \right]+   (1-\beta) L \frac{\eta_{t-1}^2}{n^2} \frac{2n^2\sigma^2}{3}\nonumber\\
    &\leq 4L^2 \eta_t^2 F_t + 4L^2\eta_{t-1}^2 F_{t-1}+  \frac{2}{3} \eta_{t-1}^2(1-\beta) L\sigma^2,
\end{align*}
where $(a)$ follows from the results of Lemma \ref{lem_bound_Bit}.
Hence the statement of Lemma \ref{lem_bound_term_2} is true for all $t \geq 2$.
\end{proof}

\section{Proofs of Lemma \ref{lem_smg_convex} and Theorem \ref{thm_smg_st_convex}}
We now provide the full proof of Lemma~\ref{lem_smg_convex} and Theorem \ref{thm_smg_st_convex} in the main text.

\subsection{Proof of Lemma \ref{lem_smg_convex}: Key bounds}
\begin{proof}
First, let us note that $\tilde{w}_{t-1} = w_{0}^{(t)}$ for $t=1, \dots, T$. From the update \eqref{update_epoch_02}, we have the following for $t\geq 2$:
\allowdisplaybreaks
\begin{align*}
    & \quad \|\tilde{w}_{t-1} - w_*\|^2 \nonumber\\
    &= \|w_{0}^{(t)} - w_*\|^2\nonumber\\
    &= \left\|w_{n}^{(t)} + \frac{\eta_t}{n}  \sum_{i=0}^{n-1} \Big( \beta g_{i}^{(t-1)} + (1-\beta)   g_{i}^{(t)} \Big)- w_*\right\|^2\nonumber\\
    &= \left\|w_{n}^{(t)} - w_*\right\|^2 + 2\frac{\eta_t}{n} \left( w_{n}^{(t)} - w_*\right)^\top   \sum_{i=0}^{n-1} \Big( \beta g_{i}^{(t-1)} + (1-\beta)   g_{i}^{(t)} \Big) \nonumber\\&\quad+\left\| \frac{\eta_t}{n}  \sum_{i=0}^{n-1} \Big( \beta g_{i}^{(t-1)} + (1-\beta)   g_{i}^{(t)} \Big)\right\|^2\nonumber\\
    &\geq  \left\|w_{n}^{(t)} - w_*\right\|^2 + 2\frac{\eta_t}{n} \left( w_{n}^{(t)} - w_*\right)^\top   \sum_{i=0}^{n-1} \Big( \beta g_{i}^{(t-1)} + (1-\beta)   g_{i}^{(t)} \Big) \nonumber\\
    &=  \left\|w_{n}^{(t)} - w_*\right\|^2 + 2\frac{\eta_t}{n}  \sum_{i=0}^{n-1} \left[ \beta \left( w_{n}^{(t)} - w_*\right)^\top   g_{i}^{(t-1)} + (1-\beta) \left( w_{n}^{(t)} - w_*\right)^\top    g_{i}^{(t)} \right] \nonumber\\
    &=\left\|w_{n}^{(t)} - w_*\right\|^2 \nonumber\\&+ \frac{2\eta_t}{n}  \sum_{i=0}^{n-1} \beta \left( f_{i}^{(t-1)} ( w_{n}^{(t)} )- f_{i}^{(t-1)}( w_*) + D_{i}^{(t-1)}(w_*; w_{i}^{(t-1)}) - D_{i}^{(t-1)}(w_{n}^{(t)}; w_{i}^{(t-1)}) \right) \nonumber\\ 
    &  + \frac{2\eta_t}{n}  \sum_{i=0}^{n-1} (1-\beta) \left( f_{i}^{(t)} ( w_{n}^{(t)} )- f_{i}^{(t)}( w_*) + D_{i}^{(t)}(w_*; w_{i}^{(t)}) - D_{i}^{(t)}(w_{n}^{(t)}; w_{i}^{(t)}) \right) ,
\end{align*}
where the last line follows from the following inequality: 
\begin{align*}
\hspace{-2ex}
    \left( w_1 - w_2\right)^\top   \nabla f_{i}^{(t)} (w_3) = f_{i}^{(t)}(w_1) - f_{i}^{(t)}(w_2) + D_{i}^{(t)} (w_2 - w_3) - D_{i}^{(t)}(w_1 - w_3).
\hspace{-2ex}    
\end{align*}

Now note that $\sum_{i=0}^{n-1} f_{i}^{(t)} (\cdot) = F (\cdot)$, we have
\begin{align*}
    \|\tilde{w}_{t-1} - w_*\|^2 
    &\geq   \left\|w_{n}^{(t)} - w_*\right\|^2 \nonumber\\&\quad+ 2\beta \eta_t [F( w_{n}^{(t)} ) - F(w_*)] \nonumber\\&\quad+ 2\frac{\eta_t}{n}  \sum_{i=0}^{n-1} \left[ \beta \left( D_{i}^{(t-1)}(w_*; w_{i}^{(t-1)}) - D_{i}^{(t-1)}(w_{n}^{(t)}; w_{i}^{(t-1)}) \right) \right]\nonumber\\
    & \quad + 2(1-\beta) \eta_t [F( w_{n}^{(t)} ) - F(w_*)] \nonumber\\&\quad+ 2\frac{\eta_t}{n}  \sum_{i=0}^{n-1} \left[ (1-\beta) \left( D_{i}^{(t)}(w_*; w_{i}^{(t)}) - D_{i}^{(t)}(w_{n}^{(t)}; w_{i}^{(t)}) \right) \right]\nonumber\\
    &\overset{\eqref{define_C}}{=}  \left\|w_{n}^{(t)} - w_*\right\|^2 + 2 \eta_t [F( w_{n}^{(t)} ) - F(w_*)] \nonumber\\&\quad+ 2\frac{\eta_t}{n}   \left[ \beta \left( C_{t-1} - \sum_{i=0}^{n-1} D_{i}^{(t-1)}(w_{n}^{(t)}; w_{i}^{(t-1)}) \right) \right]\nonumber\\
    & \quad  + 2\frac{\eta_t}{n} \left[ (1-\beta) \left( C_t- \sum_{i=0}^{n-1}D_{i}^{(t)}(w_{n}^{(t)}; w_{i}^{(t)}) \right) \right],
\end{align*}
where in the last equation we use the definition of $C_t$.
Taking expectation, we get
\begin{align*}
\hspace{-2ex}
    &\quad \E \left[\|\tilde{w}_{t-1} - w_*\|^2 \right]\nonumber\\
    &\geq  \E \left[\left\|w_{n}^{(t)} - w_*\right\|^2  \right]+ 2 \E \left[\eta_t [F( w_{n}^{(t)} ) - F(w_*)] \right] \nonumber\\
    &\quad + 2\frac{\eta_t}{n}   \left[ \beta \left( \E \left[C_{t-1} \right] - \sum_{i=0}^{n-1} \E \left[D_{i}^{(t-1)}(w_{n}^{(t)}; w_{i}^{(t-1)})  \right]\right) \right]  \nonumber\\&\quad
    + 2\frac{\eta_t}{n} \left[ (1-\beta) \left( \E \left[C_t \right]- \sum_{i=0}^{n-1}\E \left[D_{i}^{(t)}(w_{n}^{(t)}; w_{i}^{(t)})  \right]\right) \right]\nonumber\\
    &\overset{(a)}{\geq}  \E \left[\left\|w_{n}^{(t)} - w_*\right\|^2  \right]+ 2 \E \left[\eta_t [F( w_{n}^{(t)} ) - F(w_*)] \right] \nonumber\\
    &\quad + 2\frac{\eta_t}{n} F_t - 2\beta\frac{\eta_t}{n}    \left[4L^2 \eta_t^2 F_t + 4L^2\eta_{t-1}^2 F_{t-1} + \frac{2}{3} \eta_{t-1}^2(1-\beta) L\sigma^2 \right]   \nonumber\\&\quad
    - 2(1-\beta) \frac{\eta_t}{n}  \left[4 L^2\eta_t^2 F_t + \frac{2}{3}\eta_t^2  (1-\beta)L\sigma^2   \right]\nonumber\\
    &\geq  \E \left[\left\|w_{n}^{(t)} - w_*\right\|^2  \right]+ 2 \E \left[\eta_t [F( w_{n}^{(t)} ) - F(w_*)] \right] + \frac{2\eta_t}{n} F_t\nonumber\\
    &\quad  - \frac{2\eta_t}{n} \cdot 4 L^2 \eta_t^2 F_t - \frac{2\eta_t}{n} \cdot 4\beta L^2\eta_{t-1}^2 F_{t-1} \nonumber\\
    &\quad - \frac{4\eta_t}{3n}  (1-\beta) L\sigma^2 \left(\beta\eta_{t-1}^2 +(1-\beta)\eta_t^2  \right),
\hspace{-2ex}    
\end{align*}
where $(a)$ comes from the definition \eqref{define_F} of $F_t$ and the results of Lemma \ref{lem_bound_term_1} and \ref{lem_bound_term_2} for $t \geq 2$.
Note that $w_{n}^{(t)} = \tilde{w}_{t}$, $\eta_t \leq \frac{1}{2L\sqrt{K}}$ and $4 L^2 \eta_t^2 \leq \frac{1}{K}$, we further have
\begin{align*}
     2 \eta_t &\E \left[ F( \tilde{w}_{t}) - F(w_*)\right] \leq  \E \left[\|\tilde{w}_{t-1} - w_*\|^2 \right] -  \E \left[\left\|\tilde{w}_{t}- w_*\right\|^2  \right] - \frac{2\eta_t}{n} F_t \nonumber\\
     &\quad + \frac{2\eta_t}{n} \frac{1}{K} F_t + \frac{2\eta_t}{n} \beta \frac{1}{K} F_{t-1} + \frac{4L\sigma^2 }{3n}  \beta(1-\beta)  \eta_t\eta_{t-1}^2  +  \frac{4L\sigma^2}{3n} (1-\beta)^2\eta_t^3 .
\end{align*}
Now we consider the case $t=1$.
From the update \eqref{update_epoch_03}, we have the following for $t=1$:
\allowdisplaybreaks
\begin{align*}
    & \quad \|\tilde{w}_{t-1} - w_*\|^2 \nonumber\\
    &= \|w_{0}^{(t)} - w_*\|^2\nonumber\\
    &= \left\|w_{n}^{(t)} + \frac{\eta_t}{n}  \sum_{i=0}^{n-1} (1-\beta)   g_{i}^{(t)} - w_*\right\|^2\nonumber\\
    &= \left\|w_{n}^{(t)} - w_*\right\|^2 + 2\frac{\eta_t}{n} \left( w_{n}^{(t)} - w_*\right)^\top   \sum_{i=0}^{n-1}  (1-\beta)   g_{i}^{(t)}  +\left\| \frac{\eta_t}{n}  \sum_{i=0}^{n-1} \Big(  (1-\beta)   g_{i}^{(t)} \Big)\right\|^2\nonumber\\
    &\geq  \left\|w_{n}^{(t)} - w_*\right\|^2 + 2\frac{\eta_t}{n} \left( w_{n}^{(t)} - w_*\right)^\top   \sum_{i=0}^{n-1}   (1-\beta)   g_{i}^{(t)}  \nonumber\\
    &=  \left\|w_{n}^{(t)} - w_*\right\|^2 + 2\frac{\eta_t}{n}  \sum_{i=0}^{n-1} \left[ (1-\beta) \left( w_{n}^{(t)} - w_*\right)^\top    g_{i}^{(t)} \right] \nonumber\\
    &= \left\|w_{n}^{(t)} - w_*\right\|^2 \nonumber\\
     &\quad + 2\frac{\eta_t}{n}  \sum_{i=0}^{n-1} \left[ (1-\beta) \left( f_{i}^{(t)} ( w_{n}^{(t)} )- f_{i}^{(t)}( w_*) + D_{i}^{(t)}(w_*; w_{i}^{(t)}) - D_{i}^{(t)}(w_{n}^{(t)}; w_{i}^{(t)}) \right) \right],
\end{align*}
where the last line follows from the following inequality: 
\begin{align*}
    \left( w_1 - w_2\right)^\top   \nabla f_{i}^{(t)} (w_3) = f_{i}^{(t)}(w_1) - f_{i}^{(t)}(w_2) + D_{i}^{(t)} (w_2 - w_3) - D_{i}^{(t)}(w_1 - w_3).
\end{align*}
Now note that $\sum_{i=0}^{n-1} f_{i}^{(t)} (\cdot) = F (\cdot)$, we have
\begin{align*}
    \quad \|\tilde{w}_{t-1} - w_*\|^2 &\geq   \left\|w_{n}^{(t)} - w_*\right\|^2   + 2(1-\beta) \eta_t [F( w_{n}^{(t)} ) - F(w_*)] \nonumber\\
     &\quad+ 2\frac{\eta_t}{n}  \sum_{i=0}^{n-1} \left[ (1-\beta) \left( D_{i}^{(t)}(w_*; w_{i}^{(t)}) - D_{i}^{(t)}(w_{n}^{(t)}; w_{i}^{(t)}) \right) \right]\nonumber\\
    &\overset{\eqref{define_C}}{=}  \left\|w_{n}^{(t)} - w_*\right\|^2 + 2 \eta_t (1-\beta) [F( w_{n}^{(t)} ) - F(w_*)] \nonumber\\
     &\quad
    + 2\frac{\eta_t}{n} \left[ (1-\beta) \left( C_t- \sum_{i=0}^{n-1}D_{i}^{(t)}(w_{n}^{(t)}; w_{i}^{(t)}) \right) \right].
\end{align*}
Taking the full expectation, we get
\begin{align*}
    \E \left[\|\tilde{w}_{t-1} - w_*\|^2 \right]
    &\geq  \E \left[\left\|w_{n}^{(t)} - w_*\right\|^2  \right]+ 2 (1-\beta) \E \left[\eta_t [F( w_{n}^{(t)} ) - F(w_*)] \right] \nonumber\\
     &\quad
    + 2\frac{\eta_t}{n} \left[ (1-\beta) \left( \E \left[C_t \right]- \sum_{i=0}^{n-1}\E \left[D_{i}^{(t)}(w_{n}^{(t)}; w_{i}^{(t)})  \right]\right) \right]\nonumber\\
    &\overset{(a)}{\geq}  \E \left[\left\|w_{n}^{(t)} - w_*\right\|^2  \right]+ 2 (1-\beta)\E \left[\eta_t [F( w_{n}^{(t)} ) - F(w_*)] \right]\nonumber\\
     &\quad  + 2\frac{\eta_t}{n} F_t 
    - 2(1-\beta) \frac{\eta_t}{n}  \left[4 L^2\eta_t^2 F_t + \frac{2}{3}\eta_t^2  (1-\beta)L\sigma^2   \right]\nonumber\\
    &\geq  \E \left[\left\|w_{n}^{(t)} - w_*\right\|^2  \right]+ 2 (1-\beta)\E \left[\eta_t [F( w_{n}^{(t)} ) - F(w_*)] \right]\nonumber\\
     &\quad  + 2\frac{\eta_t}{n} F_t 
    - 2(1-\beta) \frac{\eta_t}{n}  \left[4 L^2\eta_t^2 F_t + \frac{2}{3}\eta_t^2  (1-\beta)L\sigma^2   \right],
\end{align*}
where $(a)$ comes from the definition \eqref{define_F} of $F_t$ and the results of Lemma \ref{lem_bound_term_1} for $t \geq 1$.
\begin{align*}
    \E \left[\|\tilde{w}_{t-1} - w_*\|^2 \right]
    &\geq  \E \left[\left\|w_{n}^{(t)} - w_*\right\|^2  \right]+ 2 (1-\beta)\E \left[\eta_t [F( w_{n}^{(t)} ) - F(w_*)] \right]\nonumber\\
     &\quad  + \frac{2\eta_t}{n} F_t 
    - 2(1-\beta) \frac{\eta_t}{n}  \left[K F_t + \frac{2}{3}\eta_t^2  (1-\beta)L\sigma^2   \right].
\end{align*}
Finally, we have
\begin{align*}
    2 \eta_t(1-\beta)\E \left[ F( \tilde{w}_{t} ) - F(w_*) \right]  &\leq  \E \left[\|\tilde{w}_{t-1} - w_*\|^2 \right] -  \E \left[\left\| \tilde{w}_{t} - w_*\right\|^2  \right] - \frac{2\eta_t}{n} F_t  \nonumber\\
    & \quad + (1-\beta) \frac{2\eta_t}{n} \frac{1}{K} F_t +  \frac{4L\sigma^2}{3n} (1-\beta)^2\eta_t^3. 
\end{align*}
This completes our proof.
\end{proof}

\subsection{Proof of Theorem \ref{thm_smg_st_convex}: Strongly convex objectives}

\begin{proof}
From the results of Lemma \ref{lem_smg_convex}, for $\eta_t \leq \frac{1}{2L\sqrt{K}}$ we have
\begin{align*}
     2 \eta_t \E \left[ F( \tilde{w}_{t}) - F(w_*)\right] &\leq  \E \left[\|\tilde{w}_{t-1} - w_*\|^2 \right] -  \E \left[\left\|\tilde{w}_{t}- w_*\right\|^2  \right] - \frac{2\eta_t}{n} F_t \nonumber\\
     + \frac{2\eta_t}{n} \frac{1}{K} F_t + \frac{2\eta_t}{n} \beta \frac{1}{K} &F_{t-1} + \frac{4L\sigma^2 }{3n}  \beta(1-\beta)  \eta_t\eta_{t-1}^2  +  \frac{4L\sigma^2}{3n} (1-\beta)^2\eta_t^3, &t \geq 2\nonumber\\
    2 \eta_t(1-\beta)\E \left[ F( \tilde{w}_{t} ) - F(w_*) \right] & \leq  \E \left[\|\tilde{w}_{t-1} - w_*\|^2 \right] -  \E \left[\left\| \tilde{w}_{t} - w_*\right\|^2  \right] - \frac{2\eta_t}{n} F_t  \nonumber\\
    &+ (1-\beta) \frac{2\eta_t}{n} \frac{1}{K} F_t +  \frac{4L\sigma^2}{3n} (1-\beta)^2\eta_t^3, &t = 1.
\end{align*}
Since $F$ is $\mu$-strongly convex, i.e. $F(w) - F(w_*) \geq \frac{\mu}{2} \| w - w_* \|^2$, for $\forall w \in \mathbb{R}^d$, we have the following:
\begin{align*}
     \mu \eta_t \E \left[\left\|\tilde{w}_{t}- w_*\right\|^2  \right] &\leq  \E \left[\|\tilde{w}_{t-1} - w_*\|^2 \right] -  \E \left[\left\|\tilde{w}_{t}- w_*\right\|^2  \right] - \frac{2\eta_t}{n} F_t &\nonumber\\
     + \frac{2\eta_t}{n} \frac{1}{K} F_t + \frac{2\eta_t}{n}& \beta \frac{1}{K} F_{t-1} + \frac{4L\sigma^2 }{3n}   \beta(1-\beta)  \eta_t\eta_{t-1}^2  +  \frac{4L\sigma^2}{3n} (1-\beta)^2\eta_t^3 , & t \geq 2\nonumber\\
    \mu \eta_t (1-\beta)\E \left[\left\|\tilde{w}_{t}- w_*\right\|^2 \right]&\leq  \E \left[\|\tilde{w}_{t-1} - w_*\|^2 \right] -  \E \left[\left\| \tilde{w}_{t} - w_*\right\|^2  \right] - \frac{2\eta_t}{n} F_t & \nonumber\\
     &+ (1-\beta) \frac{2\eta_t}{n} \frac{1}{K} F_t +  \frac{4L\sigma^2}{3n} (1-\beta)^2\eta_t^3, &t = 1.
\end{align*}
Hence, one has 
\begin{align*}
     (1+\mu \eta_t )&\E \left[\left\|\tilde{w}_{t}- w_*\right\|^2  \right] \leq  \E \left[\|\tilde{w}_{t-1} - w_*\|^2 \right]  - \frac{2\eta_t}{n} F_t  + \frac{2\eta_t}{n} \frac{1}{K} F_t\nonumber\\
     & + \frac{2\eta_t}{n} \beta \frac{1}{K} F_{t-1} + \frac{4L\sigma^2 }{3n}  \beta(1-\beta)  \eta_t\eta_{t-1}^2  +  \frac{4L\sigma^2}{3n} (1-\beta)^2\eta_t^3 , &t \geq 2,\nonumber\\
    (1+\mu \eta_t (1-\beta))&\E \left[\left\|\tilde{w}_{t}- w_*\right\|^2 \right]\leq  \E \left[\|\tilde{w}_{t-1} - w_*\|^2 \right]  - \frac{2\eta_t}{n} F_t 
    + \frac{2\eta_t}{n} \frac{1}{K} F_t \nonumber\\
     & +  \frac{4L\sigma^2}{3n} (1-\beta)^2\eta_t^3, &t = 1.
\end{align*}
Now we adopt the following conventional notations: 
\begin{align}\label{define_phi}
\phi_t := \left\{\begin{array}{ll}
1+\mu \eta_t  &\text{if} \ t \geq 2,\nonumber\\
1+\mu \eta_t (1-\beta) &\text{if} \ t = 1.
\end{array}\right.
\end{align} 
and for every $t\geq 1$:
\begin{align*}
     H_t&= - \frac{2\eta_t}{n} F_t  + \frac{2\eta_t}{n} \frac{1}{K} F_t + \frac{2\eta_t}{n} \beta \frac{1}{K} F_{t-1} + \frac{4L\sigma^2 }{3n}  \beta(1-\beta)  \eta_t\eta_{t-1}^2 +  \frac{4L\sigma^2}{3n} (1-\beta)^2\eta_t^3 ,
\end{align*}
where we use the convention that $F_0 = 0$.
Hence we have the following recursive equation for all $t\geq 1$:
\begin{align*}
    \phi_t\E \left[\left\|\tilde{w}_{t}- w_*\right\|^2  \right] &\leq  \E \left[\|\tilde{w}_{t-1} - w_*\|^2 \right] + H_t ,
\end{align*}

Unrolling this recursive equation, we have 
\begin{align*}
    \prod_{t=1}^{T} \phi_t\E \left[\left\|\tilde{w}_{T}- w_*\right\|^2  \right] &\leq 
    \prod_{t=1}^{T-1} \phi_t \left( \E \left[\left\|\tilde{w}_{T-1}- w_*\right\|^2 \right]+ H_{T} \right) \nonumber\\
    &\leq 
    \prod_{t=1}^{T-1} \phi_t  \E \left[\left\|\tilde{w}_{T-1}- w_*\right\|^2 \right]  + \prod_{t=1}^{T-1} \phi_t  H_{T}\nonumber\\ 
    &\leq 
    \prod_{t=1}^{T-2} \phi_t \E \left[\left\|\tilde{w}_{T-2}- w_*\right\|^2 \right]  +  \prod_{t=1}^{T-2} H_{T-1}+ \prod_{t=1}^{T-1} \phi_t  H_{T}\nonumber\\ 
    &\leq 
    \E \left[\left\|\tilde{w}_{0}- w_*\right\|^2 \right] + H_{1} + \phi_1 H_2 +  \prod_{t=1}^{T-2}\phi_t H_{T-1}+ \prod_{t=1}^{T-1} \phi_t  H_{T}\nonumber\\ 
    &=
    \E \left[\left\|\tilde{w}_{0}- w_*\right\|^2 \right] + \sum_{j=1}^T  H_{j}\prod_{t=1}^{j-1} \phi_t ,
\end{align*}
where 
\begin{align*}
     H_j&= \frac{2}{n} \left(\frac{1}{K} -1\right)\eta_j F_j + \frac{2}{n} \alpha \beta \frac{1}{K} \eta_{j-1}F_{j-1} + \frac{4L\sigma^2 }{3n} \alpha \beta(1-\beta) \eta_{j-1}^3  +  \frac{4L\sigma^2}{3n} (1-\beta)^2\eta_j^3 .
\end{align*}
We bound the last term: 
\begin{align*}
    \sum_{j=1}^T   H_{j}\prod_{t=1}^{j-1} \phi_t 
    & =  \frac{2}{n} \left(\frac{1}{K} -1\right) \sum_{j=1}^T  \eta_j F_j\prod_{t=1}^{j-1} \phi_t + \frac{2}{n} \alpha \beta \frac{1}{K}  \sum_{j=1}^T  \eta_{j-1}F_{j-1}\prod_{t=1}^{j-1} \phi_t\nonumber\\
    & \quad+ \frac{4L\sigma^2 }{3n} \alpha \beta(1-\beta)  \sum_{j=1}^T  \eta_{j-1}^3 \prod_{t=1}^{j-1}\phi_t +  \frac{4L\sigma^2}{3n} (1-\beta)^2  \sum_{j=1}^T  \eta_j^3\prod_{t=1}^{j-1} \phi_t \nonumber\\
    & =  \frac{2}{n} \left(\frac{1}{K} -1\right) \sum_{j=1}^T  \eta_j F_j\prod_{t=1}^{j-1} \phi_t + \frac{2}{n} \alpha \beta \frac{1}{K}  \sum_{j=0}^{T-1}  \eta_{j}F_{j}\prod_{t=1}^{j} \phi_t\nonumber\\
    & \quad+ \frac{4L\sigma^2 }{3n} \alpha \beta(1-\beta)  \sum_{j=0}^{T-1}  \eta_{j}^3 \prod_{t=1}^{j}\phi_t +  \frac{4L\sigma^2}{3n} (1-\beta)^2  \sum_{j=1}^T  \eta_j^3\prod_{t=1}^{j-1} \phi_t \nonumber\\
    & \leq \frac{2}{n}\sum_{j=1}^T  \eta_j F_j\prod_{t=1}^{j-1} \phi_t \left( \frac{1}{K} -1 + \alpha \beta \frac{1}{K} \phi_j \right) \nonumber\\
    & \quad +  \frac{4L\sigma^2}{3n}  \sum_{j=1}^T  \eta_j^3\prod_{t=1}^{j-1} \phi_t \left( (1-\beta)^2 +  \alpha \beta(1-\beta) \phi_j \right)\nonumber\\
    & \leq \frac{4L\sigma^2(1-\beta)}{3n}  \sum_{j=1}^T  \eta_j^3\prod_{t=1}^{j-1} \phi_t \left( 1-\beta +  \alpha \beta \phi_j \right),
\end{align*}
where we use the fact that $K = 1 + \alpha\beta (1 +\mu \max_t \eta_t)$.
Plugging the last bound to the previous estimate we get:
\begin{align*}
    &\quad \E \left[\left\|\tilde{w}_{T}- w_*\right\|^2  \right] \nonumber\\&  \leq \frac{1}{\prod_{t=1}^{T} \phi_t}
    \E \left[\left\|\tilde{w}_{0}- w_*\right\|^2 \right] +  \frac{1}{\prod_{t=1}^{T} \phi_t}\sum_{j=1}^T  H_{j}\prod_{t=1}^{j-1} \phi_t \nonumber\\
     &  \leq \frac{1}{\prod_{t=1}^{T} \phi_t}
    \E \left[\left\|\tilde{w}_{0}- w_*\right\|^2 \right] +   \frac{4L\sigma^2(1-\beta)}{3n}  \sum_{j=1}^T  \eta_j^3\frac{\prod_{t=1}^{j-1} \phi_t \left( 1-\beta +  \alpha \beta \phi_j \right)}{\prod_{t=1}^{T} \phi_t}\nonumber\\
     &  \leq \frac{1}{\prod_{t=1}^{T} \phi_t}
    \E \left[\left\|\tilde{w}_{0}- w_*\right\|^2 \right] +   \frac{4L\sigma^2(1-\beta)}{3n}  \sum_{j=1}^T  \eta_j^3\frac{ \left( 1-\beta +  \alpha \beta \phi_j \right)}{\prod_{t=j}^{T} \phi_t}.
\end{align*}
We use the fact that $\frac{1}{\phi_1} = \frac{1}{1 + (1-\beta) \mu \eta_1} \leq \frac{1}{(1-\beta)(1 +  \mu \eta_1)} $ and have
\begin{align*}
    &\E \left[\left\|\tilde{w}_{T}- w_*\right\|^2  \right] \nonumber\\
    &\leq  \frac{\E \left[\left\|\tilde{w}_{0}- w_*\right\|^2 \right]}{(1+(1-\beta)\mu\eta_1)\prod_{t=1}^{T-1} (1+\mu\eta_t)}
     +   \frac{4L\sigma^2}{3n}  \sum_{j=1}^T  \eta_j^3\frac{ \left( 1-\beta +  \alpha \beta (1 + \mu \max_t \eta_t)\right)}{\prod_{t=j}^{T} (1+\mu\eta_t)}\nonumber\\
    &\leq  \frac{\E \left[\left\|\tilde{w}_{0}- w_*\right\|^2 \right]}{(1-\beta) \prod_{t=1}^{T} (1+\mu\eta_t)}
     +   \frac{4L\sigma^2}{3n}  \sum_{j=1}^T  \eta_j^3\frac{ \left( 1-\beta +  \alpha \beta (1 + \mu \max_t \eta_t)\right)}{\prod_{t=j}^{T} (1+\mu\eta_t)}.
\end{align*}
This completes our proof.
\end{proof}
\end{appendices}
\bibliography{ref}
\end{document}